\documentclass[10pt]{article}
\usepackage{enumerate}
\usepackage{mathrsfs}
\usepackage{amsmath, color}
\usepackage{latexsym}
\usepackage{amssymb}
\usepackage{amsthm}
\usepackage{amsfonts}
\usepackage{dsfont}
\usepackage[hidelinks]{hyperref}
\usepackage{amsmath}
\usepackage{tikz}
\usetikzlibrary{arrows.meta}

\theoremstyle{plain}
\newtheorem{thm}{Theorem}[section]
\newtheorem{prop}[thm]{Proposition}

\newtheorem{lem}[thm]{Lemma}

\newtheoremstyle{osaka}
  {\topsep}
  {\topsep}
  {\itshape}
  {0pt}
  {\rmfamily}
  {.}
  { }
  {\thmname{#1}\thmnumber{ #2}\thmnote{ #3}}

\theoremstyle{osaka}
\newtheorem{dfn}[thm]{\bf Definition}

\newtheorem{rmk}[thm]{\bf Remark}
\newtheorem{exe}[thm]{\bf Example}

\numberwithin{equation}{section}
\newtheorem{corollary}[thm]{\bf Corollary}
\newtheorem*{theorem*}{\bf Theorem}

\newcommand{\IE}{{\mathbb{E}}}
\newcommand{\IP}{{\mathbb{P}}}
\newcommand{\IR}{{\mathbb{R}}}
\newcommand{\FF}{{\mathcal{F}}}
\newcommand{\EE}{{\mathcal{E}}}

\def\sB{\mathscr{B}}

\def\sF{\mathscr{F}}
\def\dd{\mathrm{d}}
\def\bb{\mathrm{b}}
\def\pp{\mathrm{p}}
\def\bP{\mathbb{P}}
\def\bR{\mathbb{R}}
\def\re{\mathrm{e}}
\def\ee{\mathrm{e}}
\def\bE{\mathbb{E}}
\def\sS{\mathscr{S}}
\def\sG{\mathscr{G}}
\def\cD{\mathcal{D}}
\def\sL{\mathscr{L}}
\def\sN{\mathscr{N}}
\def\bo{\mathbf{o}}

\usepackage{lipsum}

\newcommand\blfootnote[1]{%
  \begingroup
  \renewcommand\thefootnote{}\footnote{#1}%
  \addtocounter{footnote}{-1}%
  \endgroup
}

\title{\bf On hitting time distributions of Markov processes with sub-Gaussian heat kernel bounds}
\date{\today}
\author{{\bf Liping Li\footnote{The first named author is a member of LMNS,  Fudan University.  He is partially supported by NSFC (No. 12371144).}, \, Shuwen Lou}}

\begin{document}

\maketitle
\begin{abstract}
\noindent 
In this paper, we study the hitting times of Borel right processes on a metric measure space $(E,d,\mu)$ whose heat kernels satisfy sub-Gaussian bounds. It is well known that if $X=(X_t)_{t\geq 0}$ is diffusion process without killing whose heat kernel satisfies a sub-Gaussian upper bound, then, under the volume growth condition $\mu(B(x,r))\asymp r^\alpha$, it satisfies
\[
\IP^x[\tau_{B(x,r)}\le t]\le C_1\exp\left\{-C_2(r^\beta/t)^{1/(\beta -1)}\right\},
\]
where $B(x,r):=\{y\in E: d(y,x)<r\}$, $\tau_{B(x,r)}:=\inf\{t>0:X_t\notin B(x,r)\}$, and $\beta$ is the walk dimension appearing in the sub-Gaussian heat kernel estimate. We extend this result to general Borel right processes, showing that under an upper bound condition on the volume growth,
\[
\IP^x[\sigma_B\le t]\le C_3\exp \left\{-C_4\left(\frac{\widetilde d(x, B)^\beta}{t}\right)^{1/(\beta -1)}\right\},
\]
where $B$ is a nearly Borel set,  $\sigma_B$ denotes the first hitting time of $B$, and $\widetilde d(x,B)$ represents the distance from $x$ to $B$ after removing the influence of polar subsets of $B$. Furthermore, we show that for a Borel right process with a sub-Gaussian heat kernel lower bound, the hitting time distribution satisfies the corresponding lower bound
\[
\IP^x[\sigma_B\le t]\ge C_5 \exp\left\{-C_6\cdot \left(\frac{\widetilde d(x,B)^\beta}{t}\right)^{1 /(\beta -1)}\right\}.
\]
We also characterize the relationship between the constants $C_i$, $3\leq i\leq 6$, and the constants appearing in the exponents of the corresponding heat kernel bounds. As an application of these hitting time estimates, we further study the small-time asymptotic behavior of $\IP^x[\sigma_B\leq t]$ as $t\downarrow 0$.
\end{abstract}
\noindent  




\tableofcontents

\section{Introduction}

Heat kernel estimates play a fundamental role in the study of Markov processes and the analysis of their associated transition functions. Among them, Gaussian and sub-Gaussian heat kernel estimates have attracted extensive attention due to their close connections with the geometry of the underlying state spaces and the behavior of sample paths. The classic Gaussian estimates arise naturally for Brownian motions and uniformly elliptic diffusions on Euclidean domains, and have been extended to a wide range of settings, including diffusion processes on Riemannian manifolds and more general metric measure spaces. Beyond the Gaussian setting, sub-Gaussian estimates provide a natural framework for anomalous processes, where the scaling behavior is governed by a walk dimension different from the Euclidean one. Important examples include Brownian motions on fractals, such as the Sierpinski gasket and Sierpinski carpet, and symmetric diffusions associated with strongly local regular Dirichlet forms on metric measure spaces. We refer to Section~\ref{SEC71} for a more detailed discussion of processes satisfying such heat kernel estimates.

The relationship between heat kernel upper bounds and exit time estimates has been extensively studied. For diffusion processes on metric measure spaces, sub-Gaussian heat kernel upper bounds are known to imply sharp tail estimates for exit times. In particular, under suitable volume growth conditions, one has
\[
\IP^x\left[\tau_{B(x,r)}\le t\right] \lesssim \exp\left\{-C\left(\frac{r}{t^{1/\beta}}\right)^{\beta /(\beta -1)}\right\},\quad \text{for all }t>0, \,r>0.
\]
Such estimates, together with related characterizations involving heat kernel bounds and mean exit times, were established for diffusions on fractals and more general metric measure spaces; see, for example, \cite{Barlow1,Gr1,Kigami,Telcs} and the references therein. Similar characterizations for jump processes were obtained in \cite{BGK}.

In contrast to the upper bound estimates, obtaining lower bounds for hitting times from heat kernel estimates is substantially more delicate. The main difficulty lies in the fact that, intuitively,  a lower bound of the transition density only describes the probability of reaching a neighborhood of the target set at a fixed time, while a lower bound for the hitting probability requires controlling the event that the process actually visits the set before that time.

For Brownian motions on Riemannian manifolds, Grigor'yan, Saloff-Coste, and collaborators developed a method connecting hitting probabilities with heat kernels through capacity estimates. In particular, for a Brownian motion on a non-parabolic Riemannian manifold, it was shown in \cite{GSC} that, for a compact set $K$ and $x\notin K$, the distribution of the hitting time of $K$ has the following lower bound:
\[
\IP^x[\sigma_K\le t]\ge \mathrm{Cap}(K)\int_0^t \inf_{y\in \partial K} p(s,x,y)\dd s.
\]
 For parabolic manifolds, the capacity term was replaced by the conductivity between $\partial K$ and a set disjoint from $K$. This capacity-based approach was further used in \cite{BBM} to obtain sharp lower estimates for the hitting probabilities of Brownian motion in Euclidean spaces.
 More precisely, let $K\subset\mathbb{R}^d$ be a compact set of positive capacity that does not contain the origin $\mathbf{o}$, and let $a$ be a regular point of $K$. Then, for any sufficiently small $\delta>0$,
\begin{equation*}
\IP^\mathbf{o}[\sigma_K \le t]\gtrsim \ee^{-(|a|+\delta)^2/(2t)}, \quad t\in (0, T];
\end{equation*}
see \cite[Lemma 3.1]{BBM}. 
In particular, the exponential constant $1/2$ coincides with that appearing in the Gaussian heat kernel estimate.

The capacity-based arguments in \cite{GSC} provide an important inspiration for the lower bound estimates developed in this paper.

Motivated by these developments, it is natural to ask whether the connection between heat kernel estimates and hitting time distributions can be extended to a more general setting, including general Markov processes and arbitrary target sets. The present work addresses this question for Borel right processes on metric measure spaces whose transition densities satisfy sub-Gaussian estimates.

Under a sub-Gaussian heat kernel upper bound together with a volume growth condition, we establish a general upper bound estimate for the hitting time distribution; see Theorem~\ref{THM:5.1}. Compared with many existing results, which are typically restricted to metric balls or target sets with additional regularity, our estimate applies to arbitrary nearly Borel sets. Furthermore, the result is obtained in the general framework of Borel right processes, without assuming any symmetry or Dirichlet form structure. A further distinctive feature is that the estimate is expressed in terms of the modified distance $\widetilde{d}(x,B)$, rather than the usual distance $d(x,B)$, where $\widetilde{d}(x,B)$ is obtained by disregarding the polar subsets of $B$; see \eqref{eq:63}. This reflects the fundamental fact that polar subsets do not influence hitting probabilities.

The corresponding lower bound estimate of hitting time distribution under a sub-Gaussian heat kernel lower bound is substantially more involved. Unlike the upper bound estimate, where the heat kernel decay can be directly combined with volume growth arguments, a lower bound requires identifying a subset of the target set that contributes effectively to the hitting probability. The main idea is to locate, near the effective boundary of $B$, an arbitrarily small neighborhood with strictly positive capacity. The size of this neighborhood is controlled by a parameter $\delta>0$, while its location is determined by the distance after disregarding polar subsets of $B$. This modified distance, rather than the usual metric distance, plays a fundamental role in the lower bound estimate; see Lemma~\ref{LM25}. Inspired by the capacity method developed in \cite{GSC}, this approach leads to lower bound estimates under three general frameworks: the symmetric setting, the duality setting, and the sector condition setting associated with semi-Dirichlet forms; see Theorems~\ref{THM3.2} and \ref{THM:4.2}, together with their refinements in Section~\ref{SEC5}. 

As an application of the hitting time estimates established above, Section~\ref{SEC6} is devoted to the small-time asymptotic behavior of $\mathbb{P}^x(\sigma_B\leq t)$ as $t\downarrow 0$. This should be contrasted with the long-time behavior of exit and hitting probabilities, which is typically governed by spectral quantities. For instance, long-time asymptotics are closely related to the bottom of the Dirichlet spectrum for Brownian motion and more general Markov processes; see, e.g., \cite{MW}. By contrast, the small-time asymptotics are determined by the geometry of the starting point relative to the target set. In our setting, the relevant geometric quantity is the modified distance $\widetilde d(x,B)$, obtained by disregarding the polar subsets of $B$.

More generally, Section~\ref{SEC6} shows that Varadhan-type asymptotic results can be derived from the hitting time estimates established in this paper. As a representative example, suppose that the transition density satisfies the two-sided sub-Gaussian estimate
\[
    p(t,y,z)\asymp \frac{1}{t^{\alpha/\beta}}\exp \left\{-C_2\left(\frac{d(y,z)}{t^{1/\beta}}\right)^{\beta/(\beta-1)} \right\},\quad 0<t\leq T,\ y,z\in E,
\]
with the same constant $C_2$ in both the upper and lower bounds. Then
\[
 \lim_{t\downarrow 0}\frac{-t^{1/(\beta-1)}\cdot \left(\log \bP^x[\sigma_B\leq t]\right)}{C_2}= \widetilde d(x,B)^{\beta/(\beta-1)}.
 \]
Section~\ref{Sec-applications} presents several classes of examples for which the assumptions leading to this asymptotic relation are satisfied.

The remainder of the paper is organized as follows. Section~\ref{APPA} collects some preliminary results and introduces the basic setting, including the existence of transition densities and the notation for hitting time distributions. Section~\ref{S:ETUB} is devoted to the upper bound estimates of hitting time distributions derived from sub-Gaussian heat kernel upper bounds. In Section~\ref{SEC4}, lower bound estimates are established under the symmetry framework. Section~\ref{SEC5} extends these results to more general settings based on duality assumptions and the sector condition. The small-time asymptotic behavior of hitting time distributions is investigated in Section~\ref{SEC6}. Finally, Section~\ref{Sec-applications} presents several examples illustrating the applicability of the obtained results.

\subsection*{Notations} 
Let $S$ be a topological space equipped with the Borel $\sigma$-algebra $\sB(S)$. The notations $\bb \sB(S)$ and $\pp \sB(S)$ denote the classes of all bounded and nonnegative Borel measurable functions on $S$, respectively. (Similarly, for any other class of functions $\mathscr C$, we write $\bb \mathscr C$ and $\pp \mathscr C$ for the subclasses of bounded and nonnegative functions in $\mathscr C$, respectively.) For any Borel measure $\nu$ on $S$, let $\mathrm{supp}[\nu]$ denote the support of $\nu$, that is, the smallest closed set $F$ such that $\nu(F^c)=0$. Let $\mathscr{P}(S)$ denote the family of all probability measures on $(S,\sB(S))$. The completion of $\sB(S)$ with respect to a measure $\nu\in \mathscr{P}(S)$ is denoted by $\sB^\nu(S)$. Define
\[
\sB^*(S):=\bigcap_{\nu\in \mathscr{P}(S)}\sB^\nu(S),
\]
which is called the \emph{universally measurable $\sigma$-algebra} on $S$.

For two non-negative functions $f$ and $g$ on $S$, we write $f\lesssim g$ if there exists a constant $C>0$, independent of $x\in S$, such that
\[
    f(x)\leq Cg(x),\qquad x\in S.
\]
The notation $f\gtrsim g$ means that $g\lesssim f$, and $f\asymp g$ means that both $f\lesssim g$ and $f\gtrsim g$ hold. In these notations, the precise values of the multiplicative constants are not essential.

 Let $E$ be a separable metrizable space equipped with the metric $d$. Note that
\[
\begin{aligned}
\sB((0,\infty)\times E)&=\sB((0,\infty))\times \sB(E),\\
\sB((0,\infty)\times E\times E)&=\sB((0,\infty))\times \sB(E)\times \sB(E),
\end{aligned}
\]
where the right-hand sides denote product $\sigma$-algebras; see, for example, \cite[Proposition 8.1.7]{C13}. For any nonempty subset $B\subset E$, let $\overline{B}$ denote the closure of $B$, namely, the smallest closed set containing $B$, and let $\mathrm{int}(B)$ denote the interior of $B$, namely, the largest open set contained in $B$. Define $\partial B:=\overline{B}\setminus \mathrm{int}(B)\in \sB(E)$. 
For $x\in E$, define
$d(x,B):=\inf\{d(x,y):y\in B\}$.
The set $B$ is said to be $d$-bounded if there exist $o\in E$ and $r>0$ such that $d(x,o)<r$ for every $x\in B$.

\section{Preliminaries}\label{APPA}

Throughout this paper, let $E$ be a Lusin topological space equipped with the Borel $\sigma$-algebra $\sB(E)=\sigma(\bb C(E))$, and let $\mu$ be a fixed $\sigma$-finite Borel measure on $E$ with full support. Let
\begin{equation}\label{eq:21-2}
X=(\Omega,\sF,\sF_t,X_t,\theta_t,\bP^x)
\end{equation}
be a Borel right process on $E$ with lifetime $\zeta$, and denote its transition function by $(P_t)_{t\geq 0}$. The standard terminology and notation used here can be found in \cite[Appendix A2]{CF}.

\subsection{Transition density function}

We assume that $X$ satisfies the \emph{absolute continuity condition} with respect to $\mu$:
\begin{itemize}
\item[(AC)] $P_t(x,\cdot)$ is absolutely continuous with respect to $\mu$ for every $t>0$ and $x\in E$.
\end{itemize}
Since $(t,x)\mapsto P_tf(x)$ is $\sB\left((0,\infty)\times E\right)$-measurable for every $f\in \bb\sB(E)$ (see, e.g., \cite[(3.14)]{S88}), it follows from \cite[Lemma A3.2]{S88} that there exists a function
\begin{equation}\label{eq:11}
p(t,x,y),\quad t>0,x,y\in E
\end{equation}
in $\pp\sB\left((0,\infty)\times E\times E\right)$ such that
\begin{equation}\label{eq:13}
P_tf(x)=\int_E p(t,x,y)f(y)\mu(\dd y),\quad \forall f\in \bb\sB(E).
\end{equation}
The function in \eqref{eq:11} is referred to as the \emph{transition density function} (or \emph{heat kernel}) of $X$. Moreover, for any $\lambda\geq 0$,
\begin{equation}\label{potential-kernel}
u_\lambda(x,y):=\int_0^\infty \re^{-\lambda t} p(t,x,y)\dd t,\quad x,y\in E
\end{equation}
is a $\sB(E\times E)$-measurable nonnegative function (taking values in $[0,\infty]$ if $\lambda=0$). In particular, $u(x,y):=u_0(x,y)$ is called the \emph{potential kernel} of $X$.
 
We prepare the following lemma for later use. Let $\sigma$ be an $\sF_t$-stopping time. For $t>0$ and $x,y\in E$, define
\begin{equation}\label{eq:B3}
p^+_\sigma(t,x,y):=\bE^x\left[p(t-\sigma,X_\sigma,y);t>\sigma \right]. 
\end{equation}
This function gives the transition density of $X$ after the stopping time $\sigma$, as shown below.

\begin{lem}\label{PROB1}
$p^+_\sigma\in \pp \sB((0,\infty))\times \sB^*(E)\times \sB(E)$ satisfies
\begin{equation}\label{eq:B4}
    \bE^x \left[f(X_t);t>\sigma \right]=\int_E p^+_\sigma(t,x,y)f(y)\mu(\dd y),\quad \forall t>0,x\in E, f\in \bb \sB(E). 
\end{equation}
\end{lem}
\begin{proof}
We first establish the measurability of \eqref{eq:B3}. For $h(t,x,y)=h_1(t)h_2(x)h_3(y)$ with $h_1\in \bb C((0,\infty))$ and $h_2,h_3\in \bb C(E)$, we have
\[
    h(t-\sigma,X_\sigma,y)\mathbf{1}_{\{t>\sigma\}}=h_1(t-\sigma)\mathbf{1}_{\{t>\sigma\}}\cdot h_2(X_\sigma)\cdot h_3(y). 
\] 
Then
\[
H(t,x,y):=\bE^x \left[ h(t-\sigma,X_\sigma,y);t>\sigma\right]=\bE^x\left[h_1(t-\sigma)\mathbf{1}_{\{t>\sigma\}}\cdot h_2(X_\sigma) \right]\cdot h_2(y).
\]
Note that, for fixed $t$,
\[
     x\mapsto \bE^x\left[h_1(t-\sigma)\mathbf{1}_{\{t>\sigma\}}\cdot h_2(X_\sigma)  \right] \in \sB^*(E),
 \]
and, for fixed $x$,
\[
     t\mapsto \bE^x\left[h_1(t-\sigma)\mathbf{1}_{\{t>\sigma\}}\cdot h_2(X_\sigma)  \right]
 \]
is left continuous. Therefore,
\[
(t,x)\mapsto \bE^x\left[h_1(t-\sigma)\mathbf{1}_{\{t>\sigma\}}\cdot h_2(X_\sigma)\right]
\in \sB((0,\infty))\times \sB^*(E);
\]
see, e.g., \cite[(3.13)]{S88}. Consequently, $H$ is $\sB((0,\infty))\times \sB^*(E)\times \sB(E)$-measurable. A standard monotone class argument then yields that, for any
$h\in \bb\sB((0,\infty)\times E\times E)$ or $\pp\sB((0,\infty)\times E\times E)$,
\[
(t,x,y)\mapsto \bE^x\left[h(t-\sigma,X_\sigma,y);t>\sigma\right]
\]
is $\sB((0,\infty))\times \sB^*(E)\times \sB(E)$-measurable. In particular,
\[
p^+_\sigma\in \sB((0,\infty))\times \sB^*(E)\times \sB(E).
\]

The equality \eqref{eq:B4} is a consequence of the strong Markov property of $X$. Indeed, fixing $t>0$ and $x\in E$, we define
\[
    \varphi(s,\omega):=f(X_{t-s}(\omega)),\quad \text{if }0\leq s<t, \omega\in \Omega
\]
and $\varphi(s,\omega):=0$ otherwise. Then
\[
    \varphi\circ \Theta_\sigma (\omega)=\varphi(\sigma(\omega),\theta_{\sigma(\omega)}(\omega))=f(X_{t-\sigma(\omega)}(\theta_{\sigma(\omega)}(\omega)))\cdot \mathbf{1}_{\{\sigma(\omega)<t\}}=f(X_t(\omega))\cdot \mathbf{1}_{\{\sigma(\omega)<t\}}, 
\]
where $\Theta_\sigma(\omega):=(\sigma(\omega), \theta_{\sigma(\omega)}(\omega))$. Applying the strong Markov property of $X$, we obtain
\[
\begin{aligned}
    \bE^x\left[f(X_t);t>\sigma \right]&=\bE^x\left[\bE^x\left[\left(\varphi\circ \Theta_\sigma\right) \cdot \mathbf{1}_{\{\sigma<\infty\}}\big|\sF_\sigma \right]\right]  \\
    &=\bE^x \left[\bE^{z}\left[\varphi(s,\cdot) \right]\big|_{s=\sigma,z=X_\sigma}; \sigma<\infty\right].
\end{aligned}\]
By \eqref{eq:13},
\[
    \bE^z\left[\varphi (s,\cdot)\right]=\bE^z\left[f(X_{t-s}) \right]\cdot \mathbf{1}_{\{s<t\}}=\mathbf{1}_{\{s<t\}}\cdot \int_E p(t-s,z,y)f(y)\mu(\dd y). 
\]
It follows that
\[
    \bE^x\left[f(X_t);t>\sigma \right]=\bE^x\left[ \int_E p(t-\sigma,X_\sigma,y)f(y)\mu(\dd y); t>\sigma\right]=\int_E p^+_\sigma(t,x,y)f(y)\mu(\dd y).
\]
This completes the proof. 
\end{proof}

The result above immediately yields the following fact.

 \begin{corollary}
Let $p^-_\sigma(t,x,y):=p(t,x,y)-p^+_\sigma(t,x,y)$ for $t>0$ and $x,y\in E$. Then $p^-_\sigma(\cdot, \cdot,\cdot)\in \sB((0,\infty))\times \sB^*(E)\times \sB(E)$ satisfies 
    \[
        \bE^x\left[f(X_t);t\leq \sigma \right]=\int_E p^-_\sigma(t,x,y)f(y)\mu(\dd y),\quad \forall t>0,x\in E, f\in \bb\sB(E). 
    \]
In particular, 
\[
    p(t,x,y)=p^-_\sigma(t,x,y)+p^+_\sigma(t,x,y),\quad \forall t>0, x,y\in E 
\]
provides a decomposition of transition density function into the contributions before and after $\sigma$.
 \end{corollary}

 \subsection{First hitting times of nearly Borel sets}

A set $B\subset E$ is called \emph{nearly Borel measurable} (with respect to $X$) if, for each $\nu \in \mathscr{P}(E)$, there exist $B_1,B_2\in \sB(E)$ such that $B_1\subset B\subset B_2$ and  $\bP^\nu(X_t\in B_2\setminus B_1,\exists t\geq 0)=0$.
We write $\mathscr{B}^n(E)$ for the family of all nearly Borel measurable subsets of $E$, which forms a $\sigma$-algebra. It is clear that $\sB(E)\subset \sB^n(E)\subset \sB^*(E)$.  
For $B\in \sB^n(E)$, define the \emph{first hitting time} of $B$ by
\[
    \sigma_B:=\inf\{t>0:X_t\in B\}
\]
with the convention $\inf \emptyset := \infty$. We also define the \emph{first exit time} from $B$ by $\tau_B:=\inf\{t>0:X_t\notin B\}$. Note that both $\sigma_B$ and $\tau_B$ are $\sF_t$-stopping times for every $B\in \sB^n(E)$. 

Assume that $E$ is metrizable, and that its topology is induced by the metric $d$. Our main goal is to derive estimates for the hitting time distribution $\bP^x[\sigma_B\leq t]$ from sub-Gaussian estimates of the transition density function.
A key quantity in this analysis is the distance between the starting point $x$ and the target set $B$. The usual choice is the metric distance
\[
d(x,B):=\inf\{d(x,y):y\in B\}.
\]
However, for a general Borel right process, this distance may not be the optimal quantity for describing the small-time behavior of the hitting probability. Indeed, the hitting time distribution is insensitive to the modification of $B$ by polar sets, since such sets cannot be reached by the process. Therefore, removing polar subsets from $B$ does not affect $\sigma_B$, but may increase the distance between $x$ and the remaining part of $B$. To obtain sharper estimates, it is natural to replace the usual distance by a modified distance that ignores polar parts of the target set.

Let $\sN$ denote the family of all nearly Borel measurable polar sets, namely,
\[
\sN:=\{N\in\sB^n(E):\bP^x[\sigma_N<\infty]=0,\ \forall x\in E\}.
\]
For $x\in E$ and $B\in\sB^n(E)$, define
\begin{equation}\label{eq:63}
    \widetilde{d}(x,B):=\sup_{N\in\sN}d(x,B\setminus N).
\end{equation}
The lemma below gives the finiteness of $\widetilde{d}(x,B)$.

\begin{lem}
If $B\in \sB^n(E)$ is not polar, then $\widetilde{d}(x,B)<\infty$ for any $x\in E$. 
\end{lem}
\begin{proof}
  Suppose that $\widetilde{d}(x,B)=\infty$. Then, for any $n\geq 1$, there exists $N\in \sN$ such that $d(x,B\setminus N)>n$. Consequently, $B_n:=B\cap \{y\in E:d(y,x)\leq n\}$ is polar for any $n\geq 1$.  This readily indicates that $B=\bigcup_{n\geq 1}B_n$ is polar. 
\end{proof}

Clearly, $d(x,B)\leq \widetilde{d}(x,B)$, while the equality does not hold in general. In the sequel, we shall provide several sufficient conditions ensuring that $d(x,B)=\widetilde{d}(x,B)$.

A set $G\subset E$ is said to be \emph{finely open} if, for each $x\in G$, there is  a set $D(x)\in \sB^n(E)$ with $G^c\subset D(x)$ such that $\bP^x[\sigma_{D(x)}>0]=1$. The collection of all finely open subsets of $E$ forms a new topology on $E$, called the \emph{fine topology} (associated with $X$). Clearly, every open set in the original topology of $E$ is finely open, and every nonempty finely open set is non-polar. 
Recall that $\mathrm{int}(B)$ denotes the interior of $B$ with respect to the original topology of $E$. 
We use $\mathrm{int}^f(B)$ to denote the {\it fine interior} of $B$ with respect to $X$, i.e., the largest finely open set  contained in $B$.

\begin{lem}\label{LM24}
 Let $x\in E$ and $B\in \sB^n(E)$. Then the following statements hold:
\begin{itemize}
    \item[(1)] If $B$ is finely open (in particular, if $B$ is open), then $d(x,B)=\widetilde{d}(x,B)$. 
    \item[(2)] If $B\subset \overline{\mathrm{int}^f(B)}$, then $d(x,B)=\widetilde{d}(x,B)$. In particular, if $B$ is the closure of an open subset of $E$, then $d(x,B)=\widetilde{d}(x,B)$. 
\end{itemize}
\end{lem}
\begin{proof}
(1) We argue by contradiction. Suppose that $B$ is nonempty and finely open, but $d(x,B)<\widetilde{d}(x,B)$. Then, for $0<\varepsilon<\widetilde{d}(x,B)-d(x,B)$, there exists $N\in \sN$ such that $d(x,B\setminus N)>d(x,B)+\varepsilon$. Consequently,
\[
    B_\varepsilon:=\{y\in B: d(y,x)<d(x,B)+\varepsilon\}\subset N
\]
is polar. However, $B_\varepsilon$ is nonempty and finely open, since both $B$ and $\{y\in E:d(y,x)<d(x,B)+\varepsilon\}$ are finely open. This contradicts the fact that every nonempty finely open set is non-polar.

(2) Since $\mathrm{int}^f(B)\subset B\subset \overline{\mathrm{int}^f(B)}$ and $d(x,\mathrm{int}^f(B))=\widetilde{d}(x,\mathrm{int}^f(B))$ by the first statement, it follows that 
\[
    \widetilde{d}(x,B)\leq \widetilde{d}(x,\mathrm{int}^f(B))=d(x,\mathrm{int}^f(B))=d(x,\overline{\mathrm{int}^f(B)})\leq d(x,B).
\]
This gives $\widetilde{d}(x,B)=d(x,B)$. 

If $B=\overline{G}$ for some open set $G$, then $G\subset \mathrm{int}(B)$ and hence,
\[
    B=\overline{G}\subset \overline{\mathrm{int}(B)}\subset \overline{\mathrm{int}^f(B)}. 
\]
Therefore, $d(x,B)=\widetilde{d}(x,B)$. 
\end{proof}

The following fact plays a crucial role in establishing the lower bound estimate for $\bP^x[\sigma_B\leq t]$.

\begin{lem}\label{LM25}
If $B\in \sB^n(E)$ is non-polar, then for any $x\in E$ and $\delta>0$, the set
\[
    B_{x,\delta}:=\{y\in B: d(y,x)\leq \widetilde{d}(x,B)+\delta\}
\]
is also non-polar. 
\end{lem}
\begin{proof}
    If $B_{x,\delta}$ is polar for some $\delta>0$, then $\widetilde{d}(x,B)\geq d(x,B\setminus B_{x,\delta})\geq \widetilde{d}(x,B)+\delta$, which leads to a contradiction.  
\end{proof}

We conclude this section by giving a more explicit characterization of $\widetilde d(x,B)$ under typical settings. For $B\in \sB^n(E)$, let
 \[
     B^r:=\{x\in E: \bP^x[\sigma_B=0]=1\}
 \]
 denote the set of regular points  for $B$. It is well known that $B\setminus B^r$ is \emph{semipolar}.  We further point out that, under either the symmetry assumption considered in Section~\ref{SEC4} or the sector condition framework considered in Section~\ref{SEC52}, together with the absolute continuity condition (AC), every semipolar set is polar. 

\begin{prop}\label{PRO26}
Assume that every semipolar subset of $E$ is polar. Then, for any $x\in E$ and $B\in \sB^n(E)$,
\begin{equation}\label{eq:29}
    \widetilde{d}(x,B)=d(x,B\cap B^r). 
\end{equation}
\end{prop}
\begin{proof}
If $B$ is polar, then $B^r$ is empty, and \eqref{eq:29} holds trivially. Hence, we only consider the case where $B$ is not polar. 
    Note that $N_1:=B\setminus B^r$ is semipolar, and hence polar by assumption. Let $B_1:=B\cap B^r=B\setminus N_1$. Since $N_1$ is polar, we have 
    $\widetilde{d}(x,B)\geq d(x,B\setminus N_1)=d(x,B_1)$. 
    Suppose, by contradiction, that
    \[
        \widetilde{d}(x,B)=\sup_{N\in \sN}d(x,B\setminus N)=\sup_{N\in \sN}d(x,B_1\setminus N)> d(x,B_1). 
    \]
    Then, for $0<\delta<\widetilde{d}(x,B)-d(x,B_1)$, there exists $N_\delta\in \sN$ such that
    \begin{equation}\label{eq:28}
        d(x,B_1\setminus N_\delta)\geq d(x,B_1)+\delta. 
    \end{equation}  
    On the other hand, choose 
 $y\in B_1$ such that $d(x,y)\leq d(x,B_1)+\delta/3$. 
 It then follows from \eqref{eq:28} that
    \[
        B(y,\delta/3)\cap B_1\subset B_1\cap N_\delta\subset N_\delta. 
    \]
    Hence, $N_2:=B(y,\delta/3)\cap B_1$ is polar. However, $y\in N_2$ is regular for $B_1$, since
    \[
        \bP^y[\sigma_{B_1}=0]=\bP^y[\sigma_B=0]=1. 
    \]
    By the right continuity of the sample paths of $X$, $y$ is also regular for $N_2$. This contradicts the fact that $N_2$ is polar, and completes the proof of \eqref{eq:29}. 
\end{proof}

\section{Upper bound estimate of hitting time distribution
}\label{S:ETUB}

 This section is devoted to establishing an upper bound estimate for the hitting time distribution under a sub-Gaussian heat kernel upper bound. To this end, an additional volume growth condition \eqref{eq:41} is imposed. The main result is stated as follows.

\begin{thm}\label{THM:5.1}
Let $(E,d)$ be a metric space whose induced topological space is Lusin, and let $\mu$ be a $\sigma$-finite Borel measure on $E$ with full support such that, for some constants $C_0>0$ and $\alpha>0$,
\begin{equation}\label{eq:41}
    \mu(B(x, r))\le C_0 r^\alpha,\quad \forall x\in E, r>0,
\end{equation}
where $B(x,r):=\{y\in E:d(y,x)<r\}$. Consider a Borel right process $X$ on $E$ whose transition function satisfies the absolute continuity condition (AC) with respect to $\mu$. Assume further that its transition density function satisfies the following sub-Gaussian upper bound estimate: 
\begin{equation}\label{subGaussianHKUB}
p(t,x,y)\le \frac{C_1}{t^{\alpha/\beta}}\exp\left\{-C_2\left(\frac{d(x,y)}{t^{1/\beta}}\right)^{\beta/(\beta -1)}\right\}, \quad \forall  t>0, \, x,y\in E,
\end{equation}
for some $\beta>1$ and $C_1,C_2>0$. Then, for any $0<\delta<1$, there exists a constant $C_3>0$, depending only on $C_0, C_1, C_2, \alpha, \beta, \delta$ (not on $x$ or $B$), such that  
\begin{equation}\label{eq:33}
\IP^x\left[\sigma_{B}\le t\right] \le C_3 \exp\left\{-C_2(1-\delta)\left(\frac{\widetilde d(x, B)}{t^{1/\beta}}\right)^{\beta/(\beta -1)}\right\}, \quad \forall t>0, \, x\in E, B\in \sB^n(E),
\end{equation}
where $C_2$ is the same constant as in \eqref{subGaussianHKUB} and $\widetilde{d}(x,B)$ is defined as \eqref{eq:63}. 
\end{thm}
\begin{proof} 
We first note that it suffices to establish \eqref{eq:33} with $d(x,B)$ in place of $\widetilde d(x,B)$. Indeed, for any $N\in\sN$, we have $\sigma_B=\sigma_{B\setminus N}$, $\bP^x$-a.s., and hence $\bP^x[\sigma_B\leq t]=\bP^x[\sigma_{B\setminus N}\leq t]$.
If
\[
    \bP^x\left[\sigma_{B\setminus N}\leq t \right]\leq C_3 \exp\left\{-C_2(1-\delta)\left(\frac{d(x, B\setminus N)}{t^{1/\beta}}\right)^{\beta/(\beta -1)}\right\}
\]
holds for every $N\in\sN$, then taking the infimum over $N\in\sN$ on the right-hand side yields the desired estimate \eqref{eq:33}.

Fix $0<\delta<1$. Let $r:=d(x,B)=d(x,\overline{B})$. If $r=0$, then \eqref{eq:33} holds trivially. Assume that $r>0$. We have
\[
    B\subset \overline{B}\subset B(x,r)^c=\{y\in E: d(y,x)\geq r\}.
\] 
Denote by $\sigma:=\sigma_{B(x,r)^c}$ the first hitting time of $B(x,r)^c$. It follows from $\sigma\leq \sigma_B$ that
\[
    \bP^x[\sigma_B\leq t]\leq \bP^x[\sigma\leq t],\quad \forall t>0,x\in E. 
\]
Note that $\bP^x[\sigma\leq t]=\lim_{n\rightarrow \infty} \bP^x[\sigma<t+1/n]$. 
Therefore, it suffices to prove the following estimate:
\begin{equation}\label{eq:Prop2.1}
\IP^x\left[\sigma< t\right] \le C_3 \exp\left\{-C_2(1-\delta)\left(\frac{r}{t^{1/\beta}}\right)^{\beta/(\beta -1)}\right\}, \quad \forall t>0, \, x\in E
\end{equation}
for some suitable constant $C_3>0$. In what follows, we will complete the proof in several steps. In the first three steps, we fix an auxiliary parameter $\lambda>0$. It should be noted that all expressions involving $\lambda$ in these steps are valid for every fixed $\lambda>0$.

\noindent \underline{\emph{Step 1}.}  We first establish the following inequality: For any $t>0$ and $x\in E$,
\begin{equation}\label{eq2.1}
\int_E p(t,x,y)\re^{\lambda d(x,y)}\mu(\dd y) \ge  \re^{\lambda r}\left(\inf_{\substack{0< s\le t \\ z \in B(x,r)^c}} \int_E p(s, z, y) \ee^{-\lambda d(z, y)} \mu(\dd y)\right)\cdot\IP^x \left[\sigma < t \right].
\end{equation}
To derive the inequality, we start from
\begin{equation*}
\IE^x\left[f(X_t)\right]\ge \IE^x\left[f(X_t);\; \sigma<t\right]=\int_E p^+_\sigma(t,x,y)f(y)\mu(\dd y), \quad t>0, x\in E, f\in \pp \sB(E),
\end{equation*}
where $p^+_\sigma(t,x,y)=\bE^x\left[p(t-\sigma,X_\sigma,y);t>\sigma \right]$; see Lemma~\ref{PROB1}. 
Taking  $f(\cdot)=\ee^{\lambda  d(x,\cdot)}$, we obtain 
\[
    \int_E p(t,x,y)\ee^{\lambda d(x,y)}\mu(\dd y)\geq \IE^x \left[ \mathbf{1}_{\{\sigma < t\}} \int_E  p(t-\sigma, X_\sigma, y) \ee^{\lambda d(x,y)} \mu(\dd y)\right]. 
\]
Note that $d(x,y)\geq d(x,X_\sigma)-d(X_\sigma, y)\geq r-d(X_\sigma,y)$, $\bP^x$-a.s., since $X_\sigma\in B(x,r)^c$, $\bP^x$-a.s. It follows that 
\[
    \int_E p(t,x,y)\ee^{\lambda d(x,y)}\mu(\dd y)\geq \ee^{\lambda r}\, \IE^x \left[ \mathbf{1}_{\{\sigma < t\}}\int_E p(t-\sigma, X_\sigma, y)\cdot \ee^{-\lambda  d(X_\sigma, y)}\mu(\dd y) \right],
\]
which immediately implies \eqref{eq2.1}. 

\noindent\underline{\emph{Step 2}.}  Let us estimate the left-hand side of \eqref{eq2.1}. By the sub-Gaussian heat kernel estimate, we have 
\begin{equation}\label{eq2:4}
\int_E p(t,x,y) \ee^{\lambda d(x,y)}\mu(\dd y) \le \int_E \frac{C_1}{t^{\alpha/\beta}}\exp\left\{\lambda d(x,y)-C_2\left(\frac{d(x,y)}{t^{1/\beta}}\right)^{\beta/(\beta -1)}\right\} \mu(\dd y).
\end{equation}
 Write $\gamma := \beta/(\beta -1)$ for notation convenience. 
 Then 
\begin{eqnarray*}
\lambda d(x,y)-C_2\frac{d(x,y)^\gamma}{t^{\gamma -1}} =\left(\lambda d(x,y)-C_2(1-\delta/2)\frac{d(x,y)^\gamma }{t^{\gamma -1}}\right) -\frac{\delta}{2} C_2 \frac{d(x,y)^\gamma}{t^{\gamma -1}}.
\end{eqnarray*}
Consider the function
\[
    g(u):=\lambda u-Au^{\gamma}, \quad u\geq 0,
\]
where $A=C_2(1-\delta/2)/t^{\gamma -1}$. A standard calculus argument shows that  the maximum of $g$ is attained at $$u_\mathrm{max}=(\lambda/\gamma A)^{1/(\gamma -1)}$$ and
\[
    g(u_\mathrm{max})= \lambda \left(\frac{\lambda}{A\gamma}\right)^{\frac{1}{\gamma -1}} -A \left(\frac{\lambda}{\gamma A}\right)^{\frac{\gamma}{\gamma-1}} =\frac{}{}\frac{(\beta -1)^{\beta -1}\lambda ^\beta t}{\beta ^\beta(C_2 (1-\delta/2))^{\beta -1}}.
\]
In particular, 
\begin{equation*}
\lambda d(x,y)-C_2(1-\delta/2)\frac{d(x,y)^\gamma }{t^{\gamma -1}} \le \frac{(\beta -1)^{\beta -1}\lambda ^\beta t}{\beta ^\beta(C_2 (1-\delta/2))^{\beta -1}},\quad \forall x,y\in E. 
\end{equation*}
Applying the above inequality to the right-hand side of \eqref{eq2:4}, we obtain
\begin{equation}\label{eq:2.5}
\int_E p(t,x,y) \ee^{\lambda d(x,y)}\mu(\dd y)
 \le  \exp\left\{\frac{(\beta -1)^{\beta -1}\lambda ^\beta t}{\beta ^\beta(C_2 (1-\delta/2))^{\beta -1}} \right\}\int_E \frac{C_1}{t^{\alpha/\beta }}\exp\left\{-\frac{C_2 \delta}{2} \frac{d(x,y)^\gamma}{t^{\gamma -1}}\right\} \mu(\dd y). 
\end{equation}
Splitting $E$ as $E=\bigcup_{k=0}^\infty A_k$, where $A_k:=\{y\in E: k t^{1/\beta} \le d(x,y)<(k+1)t^{1/\beta}\}$, and applying \eqref{eq:41}, we have
\[
\begin{aligned}
\int_E \frac{C_1}{t^{\alpha/\beta }}\exp\left\{-\frac{C_2 \delta}{2}  \frac{d(x,y)^\gamma}{t^{\gamma -1}}\right\} \mu(\dd y)& = \sum_{k=0}^\infty \int_{A_k} \frac{C_1}{t^{\alpha/\beta }}\exp\left\{-\frac{C_2 \delta}{2}  \frac{d(x,y)^\gamma}{t^{\gamma -1}}\right\}  \mu(\dd y)
\\
& \le  \sum_{k=0}^\infty \frac{C_1}{t^{\alpha/\beta }}  \exp\left\{-\frac{C_2 \delta}{2}  k^{\beta/(\beta -1)}\right\}C_0 (k+1)^\alpha t^{\alpha/\beta} \\
&=C_0C_1 \sum_{k=0}^\infty \exp\left\{-\frac{C_2 \delta}{2}  k^{\beta/(\beta -1)}\right\}(k+1)^\alpha.
\end{aligned}
\]
Since $\beta/(\beta-1)>1$, it is straightforward to verify that
\begin{equation}\label{eq:47}
    C_\delta:=\sum_{k=0}^\infty \exp\left\{-\frac{C_2 \delta}{2}  k^{\beta/(\beta -1)}\right\}(k+1)^\alpha
\end{equation}
is finite and depends only on  $\delta$, $\alpha, \beta$ and $C_2$. 
This, together with \eqref{eq:2.5}, yields
\begin{equation}\label{eq:2.7}
\int_E p(t,x,y) \ee^{\lambda d(x,y)}\mu(\dd y) \le C_0C_1C_\delta  \exp\left\{\frac{(\beta -1)^{\beta -1}\lambda ^\beta t}{\beta ^\beta(C_2 (1-\delta/2))^{\beta -1}} \right\},\quad \forall t>0,x\in E. 
\end{equation}

\noindent\underline{\emph{Step 3}.} Next, we estimate the term inside the parentheses on the right-hand side of \eqref{eq2.1}. To this end, fix $0<s\leq t$ and $z\in B(x,r)^c$. Let $\eta>0$ be a constant to be determined later and set  $\rho: = \eta s^{1/\beta}$. For $k\ge 1$, let $B_k:=\{y\in E: k\rho \le d(z,y)<(k+1)\rho\}$. By the sub-Gaussian heat kernel upper bound and the volume growth upper bound \eqref{eq:41}, we have 
\begin{equation}\label{eq:48-2}
\begin{aligned}
 \int_{E\backslash B(z,\rho)} p(s,z, y) \mu(\dd y)
& = \sum_{k=1}^\infty \int_{B_k}p(s,z,y) \mu(\dd y)
\\
& \le  \sum_{k=1}^\infty \frac{C_1}{s^{\alpha/\beta}}\exp\left\{-\frac{C_2 (k\rho)^{\frac{\beta}{\beta -1}}}{s^{\frac{1}{\beta -1}}} \right\}C_0 (k+1)^\alpha \rho^\alpha
\\
&=C_0C_1 \eta^\alpha \sum_{k=1}^\infty \exp\left\{-C_2 k^{\frac{\beta}{\beta -1}}\eta^{\frac{\beta}{\beta -1}}\right\}(k+1)^\alpha.
\end{aligned}\end{equation}
Since $k^{\beta/(\beta-1)}\geq (k-1)^{\beta/(\beta-1)}+1$ for $k\geq 2$, the last term in \eqref{eq:48-2} is bounded by
\[
  C_0C_12^\alpha\eta^\alpha \ee^{-C_2 \eta^{\frac{\beta}{\beta -1}}}+C_0C_1\eta^\alpha \ee^{-C_2\eta^{\frac{\beta}{\beta -1}}}\sum_{k=2}^\infty \exp\left\{-C_2 (k-1)^{\frac{\beta}{\beta -1}}\eta^{\frac{\beta}{\beta -1}}\right\}(k+1)^\alpha,
\]
which converges to $0$ as $\eta\uparrow \infty$. Therefore, we can choose $\eta$, depending only on $\alpha,\beta,C_0,C_1,C_2$, such that 
\[
\int_{E\backslash B(z,\eta s^{1/\beta})} p(s,z, y) \mu(\dd y)\leq 1/2,\quad \quad \forall s>0, z\in B(x,r)^c. 
 \]
 As a result, for any $0<s\leq t$ and $z\in B(x,r)^c$,
 \begin{equation}\label{eq:48}
\int_{E} p(s,z,y) \ee^{-\lambda d(z,y)}\mu(\dd y) \ge  \int_{B(z,\eta s^{1/\beta})} p(s,z,y) \ee^{-\lambda d(z,y)}\mu(\dd y) \geq   \frac{1}{2}\ee^{-\lambda \eta s^{1/\beta}}\geq \frac{1}{2}\ee^{-\lambda \eta t^{1/\beta}}.
\end{equation}

\noindent\underline{\emph{Step 4}.} Substituting \eqref{eq:2.7} and \eqref{eq:48} into the left-hand side and the term inside the parentheses on the right-hand side of \eqref{eq2.1}, respectively, and rearranging the terms, we obtain that, for any $t>0$ and $x\in E$,  
\begin{equation}\label{eq:2.9}
\IP^x\left[\sigma < t\right] \le  2C_0C_1C_\delta \exp\left\{(\eta t^{1/\beta }- r)\cdot \lambda +\frac{(\beta -1)^{\beta -1} t}{\beta ^\beta(C_2 (1-\delta/2))^{\beta -1}} \cdot \lambda^\beta \right\},
\end{equation}
where $\eta$ is the constant chosen in the previous step. We now aim to derive \eqref{eq:Prop2.1} from \eqref{eq:2.9}. 

When $t\geq (r/\eta)^{\beta}$,  namely, $\eta t^{1/\beta}-r\geq 0$, we have
\[
    \exp\left\{-C_2(1-\delta)\left(\frac{r}{t^{1/\beta}}\right)^{\beta/(\beta -1)}\right\}\geq  \exp\left\{-C_2(1-\delta)\eta^{\beta/(\beta -1)}\right\}.
\]
Let $C^{(1)}_3:=\exp\left\{C_2(1-\delta)\eta^{\beta/(\beta -1)}\right\}$, which is a constant depending on $\alpha,\beta,C_0 , C_1,C_2$ and $\delta$. Then
\begin{equation}\label{eq:412}
    \bP^x[\sigma<t]\leq 1\leq C^{(1)}_3 \exp\left\{-C_2(1-\delta)\left(\frac{r}{t^{1/\beta}}\right)^{\beta/(\beta -1)}\right\}.
\end{equation}

When $t< (r/\eta)^{\beta}$, namely, $\eta t^{1/\beta}-r< 0$, we minimize the right-hand side of \eqref{eq:2.9} over all $\lambda>0$. Equivalently, we take
\begin{equation*}
\lambda =C_2(1-\delta/2 )\frac{\beta}{\beta -1}\left(\frac{r-\eta t^{1/\beta}}{t}\right)^{\frac{1}{\beta -1}}>0
\end{equation*} 
 to obtain
\begin{equation}\label{four-stars}
\IP^x\left[\sigma <t\right] \le  2C_0C_1C_\delta  \exp\left\{ -C_2(1-\delta/2)\left(\frac{r-\eta t^{1/\beta}}{t^{1/\beta}}\right)^{\frac{\beta}{\beta -1}}\right\}.
\end{equation}
Define
\[
    h(t):=(1-\delta)\left(\frac{r}{t^{1/\beta}}\right)^{\frac{\beta }{\beta -1}} -(1-\delta/2)\left(\frac{r-\eta t^{1/\beta}}{t^{1/\beta}}\right)^{\frac{\beta}{\beta -1}},\quad 0<t<(r/\eta)^\beta. 
\]
Then it follows from \eqref{four-stars} that
\begin{equation}\label{eq:413-2}
    \bP^x[\sigma<t]\leq 2C_0C_1C_\delta \ee^{C_2h(t)} \cdot \exp\left\{-C_2(1-\delta)\left(\frac{r}{t^{1/\beta}}\right)^{\frac{\beta }{\beta -1}}\right\}. 
\end{equation}
We claim that
\[    C_3^{(2)}:=\sup_{0<t<(r/\eta)^\beta}h(t)=\eta^{\frac{\beta}{\beta-1}}\cdot \sup_{u>1} \widetilde{h}(u),
\]
where $\widetilde{h}(u):=(1-\delta)u^{\frac{\beta }{\beta -1}} -(1-\delta/2)\left(u-1\right)^{\frac{\beta}{\beta -1}}$ for $u> 1$, is a finite constant depending on $\beta,\delta, \eta$. In fact, a standard calculus argument shows that $\widetilde{h}$ attains its maximum at 
\[
    \widetilde{u}_\mathrm{max}=\frac{1}{1-\left(\frac{1-\delta}{1-\delta/2}\right)^{\beta-1}}>1,
\]
which implies that
\begin{equation}\label{eq:413}
    C^{(2)}_3=\eta^{\frac{\beta}{\beta-1}}\widetilde h(\widetilde{u}_\mathrm{max})
\end{equation}
is finite. 

Finally, define
\[
    C_3:=\max \left\{C^{(1)}_3, 2C_0C_1C_\delta\re^{C_2C^{(2)}_3} \right\},
\]
where $C_\delta$ is given by \eqref{eq:47}, $C^{(1)}_3=\exp\left\{C_2(1-\delta)\eta^{\beta/(\beta -1)}\right\}$, and $C^{(2)}_3$ is given by \eqref{eq:413}. Clearly, $C_3$ is a constant depending only on $C_0, C_1,C_2, \alpha,\beta$, and $\delta$. Combining \eqref{eq:412} and \eqref{eq:413-2}, we finally obtain \eqref{eq:Prop2.1} with this choice of $C_3$. This completes the proof.
\end{proof}
\begin{rmk}
Since the constant $\eta$ chosen in this proof is independent of $\delta$, we have
\[
    \lim_{\delta\downarrow 0}C^{(1)}_3=\exp\{C_2 \eta^{\beta/(\beta-1)}\},
\]
while $C_\delta$ and $C_3^{(2)}$ tend to $\infty$ as $\delta\downarrow 0$. 
\end{rmk}

\section{Lower bound estimate of hitting time distribution}\label{SEC4}

From now on, we turn to establishing the lower bound estimate of $\bP^x[\sigma_B\leq t]$ for $B\in \sB^n(E)$ based on the sub-Gaussian lower bound of the transition density function.  

In this section, we impose the following additional assumption on the Borel right process $X$: 
\begin{itemize}
    \item[(H1a)] $X$ is symmetric with respect to $\mu$, namely, 
    \[
        \int_E P_t f(x)g(x)\mu(\dd x)=\int_E f(x)P_tg(x)\mu(\dd x),\quad \forall f,g\in \bb \sB(E). 
    \]
\end{itemize}
Under this symmetry assumption, the transition density function $p(t,x,y)$ is also symmetric, i.e., 
\[
    p(t,x,y)=p(t,y,x),\qquad  \text{for all } t>0,  x,y\in E.
\] 
In particular, $u_\lambda(x,y)=u_\lambda(y,x)$ for all $\lambda\geq 0$ and $x,y\in E$. 

Another consequence of (H1a) is that $X$ is associated with the so-called symmetric \emph{Dirichlet form} on $L^2(E,\mu)$:
\[
    \begin{aligned}
    & \FF:=\{f\in L^2(E,\mu): \lim_{t\downarrow 0}\EE^{(t)}(f,f)<\infty\},\\
    & \EE(f,g):=\lim_{t\downarrow 0}\EE^{(t)}(f,g),\quad f,g\in \FF,
    \end{aligned}
\]
where $\EE^{(t)}(f,g):=\frac{1}{t}(f-P_t f,g)_\mu$ for $f,g\in L^2(E,\mu)$, and $(\cdot,\cdot)_\mu$ denotes the inner product of $L^2(E,\mu)$. It is well known that $(\EE,\FF)$ is quasi-regular in the sense of \cite[Definition~1.3.8]{CF}; see, e.g., \cite[Theorem~1.5.3]{CF}. Without loss of generality (see, e.g., \cite[Theorem~1.4.3]{CF}), we make the following regularity assumption:  
\begin{itemize}
    \item[(H1b)] $(E,d)$ is a locally compact separable metric space,  $\mu$ is a positive Radon measure on $E$ with full support, $X$ is a \emph{Hunt process} (see, e.g.,  \cite[Definition A.1.23]{CF}) whose associated Dirichlet form $(\EE,\FF)$ on $L^2(E,\mu)$ is regular. 
\end{itemize}
(In fact, under the quasi-regular setting, analogous results can still be obtained; see Section~\ref{SEC52}.)

All terminology and notation related to Dirichlet forms will be kept consistent with those in \cite{CF} and \cite{FOT} as much as possible. For example, the Dirichlet form $(\EE, \FF)$ is said to be regular if $\FF\cap C_c(E)$ is dense in $\FF$ with respect to the $\EE_1$-norm 
\[
    \|f\|_{\EE_1}:=\sqrt{\EE_1(f,f)}=\sqrt{\EE(f,f)+(f,f)_\mu},\quad f\in \FF,
\] 
and dense in $C_c(E)$ with respect to the uniform norm $\|f\|_\infty:=\sup_{x\in E}|f(x)|$ for $f\in C_c(E)$, where $C_c(E)$ denotes the family of all continuous functions on $E$ with compact support. Let $\FF_\mathrm{e}$ denote the extended Dirichlet space of $(\EE,\FF)$. For convenience, every function in $\FF_\re$ is tacitly identified with its $\EE$-quasi-continuous version. 
For $\lambda>0$, denote by $\mathrm{Cap}_{(\lambda)}$ the $\lambda$-order capacity associated with $(\EE,\FF)$  (see \cite[\S2.1]{FOT}). Note that all $\lambda$-order capacities are equivalent to the $1$-order capacity. More precisely, for all $A\subset E$, $\lambda \mathrm{Cap}_{(1)}(A)\leq \mathrm{Cap}_{(\lambda)}(A)\leq \mathrm{Cap}_{(1)}(A)$ if $\lambda\leq 1$ and $\lambda^{-1} \mathrm{Cap}_{(\lambda)}(A)\leq \mathrm{Cap}_{(1)}(A)\leq \mathrm{Cap}_{(\lambda)}(A)$ if $\lambda > 1$. For convenience, the $1$-order capacity $\mathrm{Cap}_{(1)}$ will be abbreviated as $\mathrm{Cap}$. When $(\EE,\FF)$ is transient, $\mathrm{Cap}_{(0)}$ denotes the $0$-order capacity (see \cite[page 74]{FOT}). 

\subsection{Transient case}\label{S:LBHT-transient}

Let us first consider the case where $(\EE,\FF)$ is transient. This restriction will be removed in the next subsection. 
 To proceed, we prepare the following useful lemma. 

\begin{lem}\label{L:1.1}
Assume that $X$ satisfies (AC), (H1a), and (H1b), and additionally that $(\EE,\FF)$ is transient. 
Let $B\subset E$ be a nearly Borel set of finite $0$-order capacity. Then there exists a unique  finite, positive measure $\nu_B$ with $\mathrm{supp}[\nu_B]\subset \overline{B}$ such that 
\begin{equation}\label{eq:22}
    \bP^x[\sigma_B<\infty]=\int_E u(x,y)\nu_B(\dd y),\quad \forall x\in  E,
\end{equation}
where $u(x,y)$ denotes the potential kernel of $X$.
In particular,
\begin{equation}\label{eq:lem3.1}
\IP^x\left[\sigma_B\le t\right] \ge \emph{Cap}_{(0)}(B)\inf_{y\in \overline B}\int_0^t  p(s, x, y)\dd s,\quad \forall t>0,x\in E.
\end{equation}
\end{lem}
\begin{proof}
Define $p_B(x):=\bP^x[\sigma_B<\infty]$ for $x\in E$. According to \cite[Theorem 4.3.3]{FOT}, $p_B$ is an $\EE$-quasi-continuous modification of $e^{(0)}_B$, the $0$-order equilibrium potential of $B$. By the 0-order capacity version of \cite[Theorem 2.1.5]{FOT} (see \cite[page 74]{FOT}), we have
\[
    \EE(p_B,g)\geq 0,\quad \forall g\in \FF_\re \text{ with }g\geq 0,\; \EE\text{-q.e. on }B. 
\]
By \cite[Lemma 2.2.10]{FOT} and the $0$-order capacity version of \cite[Lemma 2.2.6]{FOT}, there exists a unique positive  finite measure $\nu_B$ on $E$ of finite $0$-order energy integral with $\mathrm{supp}[\nu_B]\subset \overline{B}$ such that 
\[
    p_B=U\nu_B,\; \mu\text{-a.e.},
\] 
where $U\nu_B$ denotes the $0$-order potential of $\nu_B$ (see \cite[page 85]{FOT}). 
It suffices to show that $\nu_B$ satisfies \eqref{eq:22}. 
The proof is standard, but we include the details for completeness.  Indeed,  it follows from \cite[Theorem~2.2.5]{FOT} that
\begin{equation}\label{eq:21}
    \EE(p_B,g)=\int_E g(x)\nu_B(\dd x),\quad \forall g\in \FF_\re. 
\end{equation}
Let $h$ be a reference function of the transient Dirichlet form $(\EE,\FF)$; see \cite[page 40]{FOT}. For an arbitrary $f\in \pp \sB(E)$, define
\[
    f_n:=f\wedge (nh),\quad n\geq 1. 
\]
According to \cite[Theorem 1.5.4]{FOT}, 
\[
    Uf_n(\cdot):=\int_E u(\cdot,y)f_n(y)\mu(\dd y)\in \FF_\re,\quad \EE(Uf_n,v)=(f,v)_{\mu},\; \forall v\in \FF_\re. 
\]
Taking $v=p_B$ and applying \eqref{eq:21}, we obtain
\[
    (f_n,p_B)_\mu=\EE(Uf_n,p_B)=\int_E Uf_n(x)\nu_B(\dd x)=\left(f_n, \int_E u(x,\cdot)\nu_B(\dd x)\right)_\mu.
\]
Letting $n\uparrow \infty$ yields 
\[
    p_B(\cdot)=\int_E u(x,\cdot)\nu_B(\dd x),\quad \mu\text{-a.e.}
\]
Since $u(x,y)=u(y,x)$, it follows that 
\begin{equation}\label{eq:25}
   p_B(\cdot)= \int_E u(\cdot,y)\nu_B(\dd y),\quad \mu\text{-a.e.}
\end{equation}
It is straightforward to verify that both $p_B$ and $u_B(\cdot):=\int_E u(\cdot, y)\nu_B(\dd y)$ are excessive with respect to $X$, namely, $P_t\varphi\leq \varphi$ and $\lim_{t\downarrow 0}P_t\varphi=\varphi$ pointwise for $\varphi=p_B$ or $u_B$. 
 Together with \eqref{eq:25}, this yields \eqref{eq:22}; see, e.g., \cite[Theorem~A.2.17~(iii)]{CF}.

 Fixing $x\in E$ and $t>0$, we proceed to establish \eqref{eq:lem3.1}. 
 Note that
 \[
     p_B(t,x):=\bP^x[\sigma_B\leq t]=p_B(x)-\bP^x[t<\sigma_B<\infty]. 
 \]
 Since $\sigma_B=t+\sigma_B\circ \theta_t$ on $\{\sigma_B>t\}$, it follows from the Markov property of $X$ that
 \begin{equation}\label{eq:36}
 \begin{aligned}
\bP^x[t<\sigma_B<\infty]&=\bP^x\left[\bP^x\left[\sigma_B\circ \theta_t+t<\infty,\sigma_B>t |\sF_t \right]\right] \\
 &=\bP^x\left[\bP^{X_t}[\sigma_B<\infty];\sigma_B>t \right] \\
 &=\bE^x\left[p_B(X_t);t<\sigma_B \right]\\
 &\leq \bE^x\left [p_B(X_t)\right]=\int_E p(t,x,y)p_B(y)\mu(\dd y). 
 \end{aligned}
 \end{equation}
 Using \eqref{eq:25} and \eqref{potential-kernel}, and applying the Fubini theorem, we have
 \[
 \begin{aligned}
 \int_E p(t,x,y)p_B(y)\mu(\dd y)&=    \int_E p(t,x,y)\left(\int_{\overline B} \int_0^\infty p(s,y,z)\dd s\,\nu_B(\dd z) \right)\mu(\dd y) \\
 &=\int_{\overline B}\nu_B(\dd z)\int_0^\infty \dd s \int_E p(t,x,y)p(s,y,z)\mu(\dd y) \\
 &=\int_{\overline B}\nu_B(\dd z)\int_t^\infty p(s,x,z)\dd s. 
 \end{aligned}
 \]
 Consequently, 
\[
\begin{aligned}
    p_B(t,x)&\geq p_B(x)-\int_{\overline B}\nu_B(\dd z)\int_t^\infty p(s,x,z)\dd s \\
    &=\int_{\overline B}\nu_B(\dd z)\int_0^t p(s,x,z)\dd s \\
    &\geq \nu_B(\overline B)\inf_{z\in \overline B}\int_0^t p(s,x,z)\dd s.
\end{aligned}
\]
 Finally, it suffices to note that
\[
    \nu_B(\overline B)= \int_E p_B(x)\nu_B(\dd x)=\EE(p_B,p_B)=\mathrm{Cap}_{(0)}(B). 
\]
This completes the proof.
\end{proof}
\begin{rmk}\label{RM42}
    Note that $\mathrm{Cap}_{(0)}(B)=\nu_B(\overline{B})>0$ if and only if $B$ is not polar.
\end{rmk}

The main result of this section is stated as follows. 

\begin{thm}\label{THM3.2}
Let $X$ be a Borel right process on $(E,d)$ satisfying (AC), (H1a) and (H1b), whose associated Dirichlet form $(\EE,\FF)$ on $L^2(E,\mu)$ is transient. 
Assume further that $X$ admits the following short-time heat kernel lower bound at a fixed point $x\in E$: For some $T>0, \alpha>0$ and $\beta>1$, there exist constants $C_1,C_2>0$  such that
\begin{equation}\label{SubGaussianHKLB}
p(t,x,y)\ge \frac{C_1}{t^{\alpha/\beta}}\exp\left\{-C_2\left(\frac{d(x,y)}{t^{1/\beta}}\right)^{\beta/(\beta -1)}\right\}, \quad \forall t\in (0, T], \, y\in E.
\end{equation}
Let $B\subset E$ be a nearly Borel set of finite $0$-order capacity.
Then, for any $\delta,\rho>0$, 
 \begin{equation}\label{eq:215}
        \bP^x[\sigma_B\leq t]\geq \mathrm{Cap}_{(0)}(B_{x,\delta})\cdot  \frac{C_1C_3}{ t^{(\alpha-\beta)/\beta}}\exp\left\{-(1+\rho)C_2\left( \frac{\widetilde d(x,B)+\delta}{t^{1/\beta}}\right)^{\beta/(\beta-1)}\right\},\quad 0<t\leq T,
    \end{equation}
    where $B_{x,\delta}:=\{y\in B:d(y,x)\leq \widetilde d(x,B)+\delta\}$, $C_1, C_2$ are preserved as in \eqref{SubGaussianHKLB}, and $C_3=(\beta-1)\rho\cdot \min\{(1+\rho)^{\alpha-\beta-\alpha/\beta}, 1\}>0$ is a constant depending on $\alpha, \beta,\rho$ (but not on $T$ or $\delta$).
\end{thm}
\begin{proof}
Fix $\delta, \rho>0$ and $0<t\leq T$. Note that $\sigma_B\leq \sigma_{B_{x,\delta}}$, and that $B_{x,\delta}\subset B$ is clearly of finite $0$-order capacity. It follows from \eqref{eq:lem3.1} that
\begin{equation}\label{eq:27-2}
    \bP^x[\sigma_B\leq t]\geq \bP^x[\sigma_{B_{x,\delta}}\leq t]\geq \mathrm{Cap}_{(0)}(B_{x,\delta}) \inf_{y\in \overline B_{x,\delta}} \int_0^t p(s,x,y)\dd s. 
\end{equation}
For any $y\in \overline B_{x,\delta}$, there exists $y_n\in B_{x,\delta}$ such that $\lim_{n\rightarrow \infty}d(y_n,y)=0$.  This implies $d(x,y)=\lim_{n\rightarrow \infty}d(x,y_n)\leq \widetilde d(x,B)+\delta$. Thus,
\begin{equation}\label{eq:27-3}
    p(s,x,y)\geq \frac{C_1}{s^{\alpha/\beta}}\exp\left\{-C_2\left(\frac{\widetilde d(x,B)+\delta}{s^{1/\beta}}\right)^{\beta/(\beta -1)}\right\},\quad 0<s\leq t, \quad \text{for all } y\in \overline B_{x,\delta}.
    \end{equation}
We now derive a lower bound for
\begin{equation}\label{eq:310}
  I(t):= \int_0^t \frac{1}{s^{\alpha/\beta}}\exp\left\{-C_2\left(\frac{\widetilde d(x,B)+\delta}{s^{1/\beta}}\right)^{\beta/(\beta -1)}\right\}\dd s. 
\end{equation}
For convenience, write $a:=C_2\cdot (\widetilde d(x,B)+\delta)^{\beta/(\beta-1)}>0$. Substituting $u:=as^{-1/(\beta-1)}$ into \eqref{eq:310} and noting that 
\[
  s=\left(\frac{a}{u} \right)^{\beta-1},\quad \dd s=(1-\beta)a^{\beta-1}u^{-\beta}\dd u,   
\] 
we obtain
\begin{equation}\label{eq:28-2}
    I(t)=(\beta-1)a^{-1-q}\int_{v_{a,\beta}(t)}^\infty u^q \re^{-u}\dd u,
\end{equation}
where $q:=\alpha-\beta-\frac{\alpha}{\beta}$ and $v_{a,\beta}(t)=at^{-1/(\beta-1)}$. Observe that for $v_{a,\beta}(t) \leq u\leq (1+\rho)v_{a,\beta}(t)$, 
\begin{equation}\label{eq:212}
    u^q \geq \min\{(1+\rho)^q, 1 \}\cdot v_{a,\beta}(t)^q,\quad \re^{-u}\geq \re^{-(1+\rho)v_{a,\beta}(t)}.
\end{equation}
It follows that
\[
    \int_{v_{a,\beta}(t)}^\infty u^q \re^{-u}\dd u\geq \int_{v_{a,\beta}(t)}^{(1+\rho)v_{a,\beta}(t)}u^q \re^{-u}\dd u\geq \min\{(1+\rho)^q, 1 \}\cdot \rho \cdot  v_{a,\beta}(t)^{q+1}\re^{-(1+\rho)v_{a,\beta}(t)}. 
\]
By the definition of $v_{a,\beta}(t)$, a direct computation shows that 
\[
    v_{a,\beta}(t)^{q+1}\re^{-(1+\rho)v_{a,\beta}(t)}=a^{q+1}t^{-\frac{\alpha-\beta}{\beta}}\exp\left\{-(1+\rho)C_2\left( \frac{\widetilde d(x,B)+\delta}{t^{1/\beta}}\right)^{\beta/(\beta-1)}\right\}.
\]
Combining this with \eqref{eq:28-2} and \eqref{eq:212} yields
\begin{equation}\label{eq:315}
    I(t)\geq (\beta-1)\rho \cdot \min\{(1+\rho)^{\alpha-\beta-\frac{\alpha}{\beta}}, 1 \}\cdot t^{-\frac{\alpha-\beta}{\beta}}\exp\left\{-(1+\rho)C_2\left( \frac{\widetilde d(x,B)+\delta}{t^{1/\beta}}\right)^{\beta/(\beta-1)}\right\}.
\end{equation}
Finally, we define the constant
\[
    C_3:=(\beta-1)\rho \cdot \min\{(1+\rho)^{\alpha-\beta-\frac{\alpha}{\beta}}, 1 \}>0,
\]
which depends only on $\alpha,\beta$ and $\rho$. By means of \eqref{eq:27-2}, \eqref{eq:27-3} and \eqref{eq:315}, we can eventually obtain \eqref{eq:215}. This completes the proof. 
\end{proof}

\begin{rmk}\label{RM33}
Under the same assumptions as in Theorem~\ref{THM3.2}, we establish further refinements of the lower bound estimates for $\bP^x[\sigma_B\leq t]$  as follows. Note that all the constants $C_1, C_2$ and $C_3$ below are as in \eqref{eq:215}.
\begin{itemize}
    \item[(1)] Letting $\rho=\delta$ in \eqref{eq:215}, we obtain that for any $\delta>0$, 
    \[
        \bP^x[\sigma_B\leq t]\geq \mathrm{Cap}_{(0)}(B_{x,\delta})\cdot  \frac{C_1C_3}{ t^{(\alpha-\beta)/\beta}}\exp\left\{-(1+\delta)C_2\left( \frac{\widetilde{d}(x,B)+\delta}{t^{1/\beta}}\right)^{\beta/(\beta-1)}\right\},\quad 0<t\leq T.
    \]
\item[(2)] Letting $\delta\downarrow 0$ in \eqref{eq:215}, we obtain that for any $\rho>0$, 
    \[
           \bP^x[\sigma_B\leq t]\geq \mathrm{Cap}_{(0)}(B_{x,0})\cdot  \frac{C_1C_3}{ t^{(\alpha-\beta)/\beta}}\exp\left\{-(1+\rho)C_2\left( \frac{\widetilde{d}(x,B)}{t^{1/\beta}}\right)^{\beta/(\beta-1)}\right\},\quad 0<t\leq T,
    \]
    where $B_{x,0}:=\{y\in B: d(y,x)=\widetilde{d}(x,B)\}$. 
\item[(3)] Taking the limit $\rho\downarrow 0$ in \eqref{eq:215} is not feasible, since $C_3\downarrow 0$ as $\rho\downarrow 0$. Alternatively, we can slightly modify the observation \eqref{eq:212} to eliminate the parameter $\rho$ from the lower bound estimate. This, however, comes at the expense of introducing a new constant, in place of $C_3$,
  that depends on $d(x,B)$. More precisely,  we repeat the arguments used in the proof above up to \eqref{eq:28-2}, and observe that for $v_{a,\beta}(t)\leq u\leq v_{a,\beta}(t)+1$,
 \[
\re^{-u}\geq \re^{-v_{a,\beta}(t)-1},
 \]
 and
 \[
     u^q \geq \min\left\{\left(1+\frac{1}{v_{a,\beta}(T)}\right)^q, 1 \right\}\cdot v_{a,\beta}(t)^q,\quad 0<t\leq T. 
 \]
 This implies that for any $0<t\leq T$, 
 \[
 \begin{aligned}
     I(t)&\geq (\beta-1)a^{-1-q}\int_{v_{a,\beta}(t)}^{v_{a,\beta}(t)+1} u^q \re^{-u}\dd u \\
     &\geq (\beta-1)a^{-1-q}\re^{-1}\min\left\{\left(1+\frac{1}{v_{a,\beta}(T)}\right)^q, 1 \right\}\cdot v_{a,\beta}(t)^q\re^{-v_{a,\beta}(t)}.
 \end{aligned}\]
 Therefore, we obtain that 
 \[
     \bP^x[\sigma_B\leq t]\geq \mathrm{Cap}_{(0)}(B_{x,\delta})\cdot \frac{C'_3}{ t^{\frac{\alpha-\beta}{\beta}-\frac{1}{\beta-1}}}\exp\left\{-C_2\left( \frac{\widetilde{d}(x,B)+\delta}{t^{1/\beta}}\right)^{\beta/(\beta-1)}\right\},\quad 0<t\leq T,
 \]
 where 
 \[
     C'_3=
\frac{C_1}{C_2}(\beta-1)\re^{-1}\cdot \left(\widetilde{d}(x,B)+\delta\right)^{-\beta/(\beta-1)} \cdot  \min\left\{\left(1+\frac{T^{1/(\beta-1)}}{C_2\left(\widetilde{d}(x,B)+\delta\right)^{-\beta/(\beta-1)}}\right)^{\alpha-\beta-\alpha/\beta}, 1 \right\}
 \]
 is a positive constant depending on $T,\alpha,\beta,\delta$ and $\widetilde{d}(x,B)$. 
\end{itemize}
The following remarks concern the capacity appearing in the lower bound estimate \eqref{eq:215}.
\begin{itemize}
    \item[(4)] If the metric $d$ on $E$ satisfies that the closure of every $d$-bounded subset of $E$ is compact (a metric compatible with the topology of $E$ and satisfying this property always exists; see, e.g., \cite[Chapter 0, (3.1)]{BG68}), then the condition that $B$ is of finite $0$-order capacity can be removed. Indeed, $B_{x,\delta}$ is $d$-bounded, so its closure is compact; in particular, $B_{x,\delta}$ is of finite $0$-order capacity.
    \item[(5)] By Lemma~\ref{LM25} and Remark~\ref{RM42}, $\mathrm{Cap}_{(0)}(B_{x,\delta})>0$ for any $\delta>0$ provided that $B$ is not polar. 
It is also worth noting that $\mathrm{Cap}_{(0)}(B_{x,0})$ can possibly be strictly positive. For example, when $(\EE,\FF)$ corresponds to an irreducible and transient diffusion process on $\bR$ and $B=\{y\}$, we have $B_{x,0}=\{y\}$ and $\mathrm{Cap}_{(0)}(B_{x,0})=\mathrm{Cap}_{(0)}(\{y\})>0$ (see, e.g., \cite[Example~2.1.2]{FOT}). 
\end{itemize}
\end{rmk}

By virtue of \eqref{eq:215}, we can further obtain the following lower bound estimates for $\bP^x[\sigma_B\leq t]$. 

\begin{corollary}\label{COR35}
Under the same assumptions as in Theorem~\ref{THM3.2}, the following estimates hold true: for any $\delta,\rho>0$, 
\begin{equation}\label{eq:317}
    \IP^x[\sigma_B\le t]\ge \mathrm{Cap}_{(0)}(B_{x,\delta/2})\cdot \widetilde C_3\cdot \exp\left\{ -C_2 (1+\rho) \left( \frac{\widetilde{d}(x,B)+\delta}{t^{1/\beta}}\right)^{\beta/(\beta -1)}\right\}, \quad \forall 0<t\le T,
\end{equation}
where $\widetilde C_3>0$ is a constant depending on $T,\alpha, \beta, \delta$ and $\rho$. 
\end{corollary}
\begin{proof}
Replacing $\delta$ by $\delta/2$ in \eqref{eq:215}, we have 
\[
    \bP^x[\sigma_B\leq t]\geq \mathrm{Cap}_{(0)}(B_{x,\delta/2})\cdot  \frac{C_1C_3}{ t^{(\alpha-\beta)/\beta}}\exp\left\{-C_2(1+\rho)\left( \frac{\widetilde{d}(x,B)+\delta/2}{t^{1/\beta}}\right)^{\beta/(\beta-1)}\right\}.
\]
Since $\beta/(\beta-1)>1$, applying $(b+c)^{\beta/(\beta-1)}\geq b^{\beta/(\beta-1)}+c^{\beta/(\beta-1)}$  for $b,c\geq 0$ yields
\[
    (\widetilde{d}(x,B)+\delta/2)^{\beta/(\beta-1)}\leq (\widetilde{d}(x,B)+\delta)^{\beta/(\beta-1)}- (\delta/2)^{\beta/(\beta-1)}.
\]
This implies that
\[
\begin{aligned}
   \exp& \left\{-C_2(1+\rho)\left( \frac{\widetilde{d}(x,B)+\delta/2}{t^{1/\beta}}\right)^{\beta/(\beta-1)}\right\} \\
   &\geq \exp\left\{C_2(1+\rho)(\delta/2)^{\beta/(\beta-1)}t^{-1/(\beta-1)}\right\}\cdot \exp\left\{-C_2(1+\rho)\left( \frac{\widetilde{d}(x,B)+\delta}{t^{1/\beta}}\right)^{\beta/(\beta-1)}\right\}. 
\end{aligned}
\]
Define
\[
   f(t):= f_{T,\alpha,\beta,\delta,\rho}(t)=\frac{1}{ t^{(\alpha-\beta)/\beta}}\exp\left\{C_2(1+\rho)(\delta/2)^{\beta/(\beta-1)}t^{-1/(\beta-1)}\right\},\quad t>0. 
\]
If $\alpha\geq \beta$, then $f$ is decreasing in $t$, and 
\[
    f(t)\geq f(T)>0,\quad \forall 0<t\leq T.  
\]
If $\alpha<\beta$, then a standard argument shows that 
\[
    f(t)\geq f(t_{\mathrm{min}})>0,\quad \forall t>0,
\]
where
\[
t_{\mathrm{min}}=\left(\frac{C_2(1+\rho)(\delta/2)^{\beta/(\beta-1)}}{(\beta-1)(\beta-\alpha)/\beta} \right)^{\beta-1}>0. 
\]
Finally, defining the constant
\[
\widetilde C_3:=C_1C_3 \cdot  \min\{f(T), f(t_{\mathrm{min}}) \}>0,
\]
which depends on $T,\alpha,\beta,\delta$ and $\rho$, we arrive at \eqref{eq:317}. This completes the proof. 
\end{proof}

\begin{rmk}\label{RM36}
\begin{itemize}
\item[(1)] Letting $\rho=\delta$ in \eqref{eq:317}, we obtain that for any $\delta>0$,
\[
\IP^x[\sigma_B\le t]\ge \mathrm{Cap}_{(0)}(B_{x,\delta/2})\cdot \widetilde{C}'_3\cdot \exp\left\{ -C_2 (1+\delta) \left( \frac{\widetilde{d}(x,B)+\delta}{t^{1/\beta}}\right)^{\beta/(\beta -1)}\right\}, \quad \forall 0<t\le T,
\]
where $\widetilde{C}'_3>0$ is a constant depending on $T,\alpha,\beta$ and $\delta$. 
\item[(2)] Mimicking the arguments in Remark~\ref{RM33}~(3), we can also eliminate the parameter $\rho$ from \eqref{eq:317} and obtain that for any $\delta>0$,
\[
      \IP^x[\sigma_B\le t]\ge \mathrm{Cap}_{(0)}(B_{x,\delta/2})\cdot \widetilde C''_3\cdot \exp\left\{ -C_2 \left( \frac{\widetilde{d}(x,B)+\delta}{t^{1/\beta}}\right)^{\beta/(\beta -1)}\right\}, \quad \forall 0<t\le T,
\]
where $\widetilde C''_3>0$ is a constant depending on $T,\alpha, \beta, \delta$ and $\widetilde{d}(x,B)$.
\end{itemize} 
\end{rmk}

\subsection{General case}\label{S:LBHT-general}

We now proceed to remove the transient assumption imposed in the previous subsection. 

\begin{lem}\label{LM41}
Let $X$ be a Borel right process on $(E,d)$ satisfying (AC), (H1a) and (H1b).
Let $B\subset E$ be a nearly Borel set of finite $1$-order capacity. Then for any $\lambda>0$, there exists a unique finite, positive measure $\nu^{(\lambda)}_B$ on $E$ with $\mathrm{supp}[\nu_B^{(\lambda)}]\subset \overline{B}$ such that
\begin{equation}\label{eq:42}
     p_B^{(\lambda)}(x):=\bE^x[\re^{-\lambda \sigma_B}]=\int_E u_\lambda(x,y)\nu_B^{(\lambda)}(\dd y),\quad \forall x\in E.
 \end{equation}
In particular, for any $\lambda>0$, 
\[
\IP^x\left[\sigma_B\le t\right] \ge \emph{Cap}_{(\lambda)}(B)  \inf_{y\in \overline B} \int_0^t \re^{-\lambda s} p(s, x, y)\dd s,\quad \forall t>0,x\in E.
\]
\end{lem}
\begin{proof}
 Fix $\lambda>0$.
Mimicking the argument in the proof of Lemma~\ref{L:1.1}, we obtain a positive Radon measure $\nu_B^{(\lambda)}$ with $\mathrm{supp}[\nu_B^{(\lambda)}]\subset \overline{B}$ such that $p_B^{(\lambda)}=U_\lambda \nu_B^{(\lambda)}$, $\mu$-a.e., where $U_\lambda \nu_B^{(\lambda)}$ is the $\lambda$-order potential of $\nu_B^{(\lambda)}$. As in the proof of Lemma~\ref{L:1.1}, we can similarly demonstrate that $\nu^{(\lambda)}_B$ is the desired measure.

    Note that for all $x\in E$,
    \begin{equation}\label{eq:44}
        \bP^x\left[\sigma_B\leq t \right]\geq \bE^x\left[\re^{-\lambda \sigma_B};\sigma_B\leq t \right]=\bE^x[\re^{-\lambda \sigma_B}]-\bE^x[\re^{-\lambda \sigma_B};\sigma_B>t]. 
    \end{equation}
Mimicking \eqref{eq:36} and using the Markov property, we have
 \begin{equation}\label{eq:43}
     \bE^x[\re^{-\lambda \sigma_B};\sigma_B>t]=\re^{-\lambda t}\bE^x \left[p_B^{(\lambda)}(X_t);\sigma_B> t \right]\leq \re^{-\lambda t}\int_E p(t, x,y)p_B^{(\lambda)}(y)\mu(\dd y). 
 \end{equation}
Substituting \eqref{potential-kernel} into \eqref{eq:42}, then \eqref{eq:42} into \eqref{eq:43} and applying Fubini's theorem, we get
 \[
 \begin{aligned}
      \bE^x[\re^{-\lambda \sigma_B};\sigma_B>t]&\leq \int_E \re^{-\lambda t}p(t,x,y)\mu(\dd y)\int_E \int_0^\infty \re^{-\lambda s}p(s, y,z)\dd s\,\nu_B^{(\lambda)}(\dd z) \\
      &=\int_E \nu_B^{(\lambda)}(\dd z)\int_0^\infty \re^{-\lambda(t+s)}\dd s\int_E p(t, x,y)p(s, y,z)\mu(\dd y) \\
      &=\int_E \nu_B^{(\lambda)}(\dd z)\int_0^\infty \re^{-\lambda(t+s)}p(t+s, x,z)\dd s \\
      &=\int_E \nu_B^{(\lambda)}(\dd z)\int_t^\infty \re^{-\lambda s}p(s,x,z)\dd s.
 \end{aligned}\]
 It then follows from \eqref{eq:44} and $\mathrm{supp}[\nu_B^{(\lambda)}]\subset \overline{B}$ that 
 \[
     \bP^x[\sigma_B\leq t]\geq \int_{\overline{B}}\nu_B^{(\lambda)}(\dd y)\int_0^t \re^{-\lambda s}p(s, x,y)\dd s\geq \nu_B^{(\lambda)}(\overline{B})\inf_{y\in \overline{B}}\int_0^t \re^{-\lambda s}p(s,x,y)\dd s. 
 \]
 Finally, it suffices to note that $\nu_B^{(\lambda)}(\overline{B})=\mathrm{Cap}_{(\lambda)}(B)$, which completes the proof.
\end{proof}
\begin{rmk}
As in the transient case (see Remark~\ref{RM42}), $\nu^{(\lambda)}_B(\overline{B})=\mathrm{Cap}_{(\lambda)}(B)>0$ if and only if $B$ is not polar. 
\end{rmk}

By utilizing Lemma~\ref{LM41} in place of Lemma~\ref{L:1.1}, we readily obtain the following result. 

\begin{thm}\label{THM:4.2}
Let $X$ be a Borel right process on $(E,d)$ satisfying (AC), (H1a) and (H1b).
 Further assume that $X$ admits the following short-time heat kernel lower bound at a fixed point $x\in E$: For some $T>0, \alpha>0,\beta>1$, there exist constants $C_1,C_2>0$ such that
\begin{equation}\label{eq:421}
p(t,x,y)\ge \frac{C_1}{t^{\alpha/\beta}}\exp\left\{-C_2\left(\frac{d(x,y)}{t^{1/\beta}}\right)^{\beta/(\beta -1)}\right\}, \quad \forall t\in (0, T], \, y\in E.
\end{equation}
Let $B\subset E$ be a nearly Borel set of finite $1$-order capacity.
Then for any $\lambda, \delta,\rho>0$, 
 \begin{equation}\label{eq:45}
        \bP^x[\sigma_B\leq t]\geq \frac{\mathrm{Cap}_{(\lambda)}(B_{x,\delta})}{\re^{\lambda T}}\cdot  \frac{C_1C_3}{ t^{(\alpha-\beta)/\beta}}\exp\left\{-(1+\rho)C_2\left( \frac{\widetilde{d}(x,B)+\delta}{t^{1/\beta}}\right)^{\beta/(\beta-1)}\right\},\quad 0<t\leq T,
    \end{equation}
    where $B_{x,\delta}:=\{y\in B:d(y,x)\leq \widetilde{d}(x,B)+\delta\}$, $C_1$ and $C_2$ are as in \eqref{eq:421}, and $C_3=(\beta-1)\rho\cdot \min\{(1+\rho)^{\alpha-\beta-\alpha/\beta}, 1\}>0$ is a constant depending on $\alpha, \beta$ and $\rho$.
\end{thm}
\begin{proof}
    Mimicking \eqref{eq:27-2}, we  obtain that for any $0<t\leq T$, 
    \[
        \bP^x[\sigma_B\leq t]\geq \mathrm{Cap}_{(\lambda)}(B_{x,\delta}) \inf_{y\in \overline B_{x,\delta}} \int_0^t \re^{-\lambda s} p(s, x, y)\dd s\geq \frac{\mathrm{Cap}_{(\lambda)}(B_{x,\delta})}{\re^{\lambda T}} \inf_{y\in \overline B_{x,\delta}} \int_0^t  p(s, x, y)\dd s. 
    \]
    Then repeating the argument in the proof of Theorem~\ref{THM3.2} can complete the current proof. 
\end{proof}

It is worth pointing out that all the refinements stated in Remark~\ref{RM33}, Corollary~\ref{COR35} and Remark~\ref{RM36}, with $\mathrm{Cap}_{(\lambda)}(B_{x,\delta})/\re^{\lambda T}$ in place of $\mathrm{Cap}_{(0)}(B_{x,\delta})$, remain valid for the estimate \eqref{eq:45}.  

\section{Refinements of the symmetry assumption for lower bound estimates}\label{SEC5}


In the analysis carried out in Section~\ref{SEC4}, the symmetry assumption (H1a), together with the regularity assumption (H1b), is invoked essentially only once, namely, in establishing the integral representation \eqref{eq:42} for hitting probabilities.
We will consider several commonly used alternative settings that can replace  the symmetry assumption. 

To proceed, we introduce some notation and terminology. For $\lambda\geq 0$, $f\in \pp \sB(E)$ and $B\in \sB^n(E)$, define for all $x\in E$,
\[
U_\lambda f(x):=\int_E u_\lambda(x,y)f(y)\mu(\dd y),
\]
and
\[
    P_Bf(x):=\bE^x\left[f(X_{\sigma_B})\right],\quad p^{(\lambda)}_B(x):=\bE^x\left[\ee^{-\lambda \sigma_B};\sigma_B<\infty\right].
\]
Note that $p^{(\lambda)}_B$ is $\lambda$-excessive with respect to $X$.
 For any positive measure $\nu$ on $E$, define the measure $\nu U_\lambda$ by $\nu U_\lambda(B):=\int_E U_\lambda 1_B(x)\nu(\dd x)$.  For convenience, write $Uf:=U_0f, p_B:=p^{(0)}_B$ and $\nu U:=\nu U_0$. 
The Borel right process $X$ is called \emph{transient} if there exists a strictly positive function $h\in \sB^*(E)$ such that $Uh(x)<\infty$ for all $x\in E$. 

\subsection{Refinement with duality}\label{subsection:5.1}

Now we consider a  more general setting arises where $X$ admits a dual process with respect to $\mu$:
\begin{itemize}
    \item[(H2a)] There exists another Borel right process $\widehat{X}$ on $E$, with transition function $(\widehat{P}_t)_{t\geq 0}$, that satisfies (AC) with respect to $\mu$ and 
 \[
     \int_E P_t f(x)g(x)\mu(\dd x)=\int_E f(x)\widehat P_t g(x)\mu(\dd x),\quad \forall f,g\in \bb\sB(E). 
 \]
\end{itemize}
Clearly, (H1a) implies (H2a) with $\widehat{X}=X$. Under the duality assumption (H2a),  $\widehat P_t$ admits the transition density function
\[
    \widehat{p}(t,x,y)=p(t,y,x),\quad \forall t>0,x,y\in E, 
\]
and the nearly Borel $\sigma$-algebra with respect to $\widehat{X}$ coincides with $\sB^n(E)$, the nearly Borel $\sigma$-algebra with respect to $X$ (see \cite[Chapter VI, (1.20)]{BG68}).
In addition, in place of the regularity condition (H1b), we need the following condition:
\begin{itemize}
    \item[(H2b)] $(E,d)$ is a locally compact separable metric space, $\mu$ is a Radon measure on $E$; and for $\lambda>0$ ($\lambda\geq 0$ if $X$ is transient) and $f\in C_c(E)$, 
    \begin{equation}\label{eq:51-2}
        y\mapsto \int_E u_\lambda(x,y)f(x)\mu(\dd x)
    \end{equation}
    is continuous on $E$. 
\end{itemize}
Note that if $\widehat{P}_t(C_c(E))\subset C(E)$ for all $t>0$, then the second condition in (H2b) automatically holds for $\lambda>0$; that is, \eqref{eq:51-2} is continuous. L\'evy processes on Euclidean spaces provide a typical class of processes satisfying both (H2a) and (H2b). Indeed, for any L\'evy process $X$, its dual process with respect to Lebesgue measure is given by $-X$, which is itself a L\'evy process. Consequently, $-X$ is a Feller process and, in particular, satisfies (H2b).


\begin{lem}\label{LM51}
Assume that the Borel right process $X$ satisfies  (AC), (H2a) and (H2b). Let $B\in \sB^n(E)$ be such that $\overline{B}$ is compact in $E$. Then for any $\lambda>0$ ($\lambda\geq 0$ if $X$ is transient), there exists a unique positive  measure $\nu^{(\lambda)}_B$ on $E$ such that 
\[
   p_B^{(\lambda)}(x)
   =\int_E u_\lambda(x,y)\nu^{(\lambda)}_B(\dd y),\quad \forall x\in E. 
\]
Moreover, $\nu^{(\lambda)}_B$ is finite and concentrated on $\overline{B}$. 
\end{lem}
\begin{proof}
The uniqueness follows from \cite[Chapter VI, (1.15)]{BG68} and relies  only on assumption (H2a). To establish existence, 
    we first consider the case $\lambda=0$ under the assumption that $X$ is transient. 
 The desired conclusion is essentially a modification of 
 \cite[Chapter V, Theorem 2.8]{BG68}, and we defer the proof details to Lemma~\ref{LMB1}. 

    For $\lambda>0$, we consider the subprocess $X^{(\lambda)}$ of $X$ with transition function $(\ee^{-\lambda t}P_t)_{t\geq 0}$, whose potential kernel is $u_\lambda(x,y)$. Let $(k_t)_{t\geq 0}$ be the killing operators of $X$, as defined in \cite[Definition (11.3)]{S88}, and define $m_t:=\ee^{-\lambda t}1_{[0,\zeta)}(t)$ for $t\geq 0$, where $\zeta$ is the lifetime of $X$.  According to \cite[Theorem (61.5)]{S88}, $X^{(\lambda)}$ can be realized as
    \[
        X^{(\lambda)}=(\Omega, \hat\sF, \hat\sF_t,X_t,\theta_t,\hat{\bP}^x)
    \]
    with state space $E$, where $\Omega, X_t,\theta_t$ are as in \eqref{eq:21-2}, and $\hat{\bP}^x$ satisfies 
    \begin{equation}\label{eq:52}
        \hat{\bE}^x[H]=\bE^x \left[\int_0^\infty H\circ k_t\, (-\dd m_t)\right], \quad \forall H\in \bb \sF^*,
    \end{equation}
    where $\sF^*$ is the universal completion of $\sF^0:=\sigma(X_t:t\geq 0)$; see \cite[(61.2)]{S88}. Note that $\{\sigma_B<\infty\}\in \sF^*$; see \cite[page 52]{S88}. Taking $H=\mathbf{1}_{\{\sigma_B<\infty\}}=\mathbf{1}_{\{\sigma_B<\zeta\}}$ in \eqref{eq:52} and noting that $H\circ k_t(\omega)=\mathbf{1}_{\{\sigma<t\}}$, we obtain
    \begin{equation}\label{eq:54}
    \begin{aligned}
    \hat{\bP}^x\left[\sigma_B<\infty\right]&=\bE^x \left[\int_{(0,\zeta)}\mathbf{1}_{\{\sigma<t\}} (-\dd (\ee^{-\lambda t})\right]+\bE^x\left[\ee^{-\lambda \zeta};\sigma_B<\zeta \right] \\
    &=\bE^x\left[\ee^{-\lambda \sigma_B};\sigma_B<\infty \right]. 
    \end{aligned}
    \end{equation}
    Clearly, $X^{(\lambda)}$ is transient. Applying the argument of the previous paragraph to $\hat{\bP}^x[\sigma_B<\infty]$ with respect to $X^{(\lambda)}$, we obtain the desired measure $\nu^{(\lambda)}_B$. This completes the proof. 
\end{proof}

Under the current setting, we can establish lower bound estimates for the hitting time distribution analogous to those in Theorems \ref{THM3.2} and \ref{THM:4.2}. The proofs in Section~\ref{SEC4} remain valid with only one modification: the term $\mathrm{Cap}_{(\lambda)}(B_{x,\delta})$ should be replaced by $\nu_B^{(\lambda)}(\overline{B}_{x,\delta})$.
In this setting, the assumption that $B$ has finite capacity is replaced by the compactness of $\overline{B}$, which ensures that $\nu_B^{(\lambda)}(\overline{B}_{x,\delta})$ is finite. This compactness assumption can be further removed if every $d$-bounded closed subset of $E$ is compact; see Remark~\ref{RM33}~(4).


    \begin{thm}\label{THM52}
Let $X$ be a Borel right process on $(E,d)$ that satisfies (AC), (H2a) and (H2b).
 Further assume that $X$ admits the following short-time heat kernel lower bound at a fixed point $x\in E$: For some $T>0, \alpha>0,\beta>1$, there exist constants $C_1,C_2>0$ such that
\begin{equation}\label{eq:55}
p(t,x,y)\ge \frac{C_1}{t^{\alpha/\beta}}\exp\left\{-C_2\left(\frac{d(x,y)}{t^{1/\beta}}\right)^{\beta/(\beta -1)}\right\}, \quad \forall t\in (0, T], \, y\in E.
\end{equation}
Let $B\subset E$ be a nearly Borel set with compact closure $\overline{B}$. 
Then for any $\delta,\rho>0$ and any $\lambda>0$ ($\lambda\geq 0$ if $X$ is transient),
 \[
        \bP^x[\sigma_B\leq t]\geq \frac{\nu^{(\lambda)}_{B_{x,\delta}}(\overline{B}_{x,\delta})}{\re^{\lambda T}}\cdot \frac{C_1C_3}{ t^{(\alpha-\beta)/\beta}}\exp\left\{-(1+\rho)C_2\left( \frac{\widetilde{d}(x,B)+\delta}{t^{1/\beta}}\right)^{\beta/(\beta-1)}\right\},\quad 0<t\leq T,
    \]
    where  $B_{x,\delta}:=\{y\in B:d(y,x)\leq \widetilde{d}(x,B)+\delta\}$, $\nu^{(\lambda)}_{B_{x,\delta}}$ is the measure obtained in Lemma~\ref{LM51}, $C_1$ and $C_2$ are as in \eqref{eq:55}, and $C_3=(\beta-1)\rho\cdot \min\{(1+\rho)^{\alpha-\beta-\alpha/\beta}, 1\}>0$ is a constant depending on $\alpha, \beta,\rho$.
\end{thm}
\begin{proof}
Note that $B_{x,\delta}\in \sB^n(E)$ since $B\in \sB^n(E)$, and its closure is compact since $\overline{B}$ is compact. 
Mimicking the proof of Lemma~\ref{LM41} for $\lambda>0$ (and the last step of the proof of Lemma~\ref{L:1.1} for $\lambda=0$), and applying Lemma~\ref{LM51} to $B_{x,\delta}$, we obtain
\[
     \bP^x[\sigma_{B_{x,\delta}}\leq t]\geq  \nu^{(\lambda)}_{B_{x,\delta}}(\overline{B}_{x,\delta})\inf_{y\in \overline{B}_{x,\delta}}\int_0^t \re^{-\lambda s}p(s,x,y)\dd s\geq \frac{\nu^{(\lambda)}_{B_{x,\delta}}(\overline{B}_{x,\delta})}{\ee^{\lambda T}}\inf_{y\in \overline{B}_{x,\delta}}\int_0^t p(s,x,y)\dd s,\quad \forall 0<t\leq T. 
 \]
 The remaining argument is the same as that in the proof of Theorem~\ref{THM3.2}. 
\end{proof}

\subsection{Refinement with sector condition}\label{SEC52}

Another way to generalize the symmetry assumption (H1a) is to impose a sector condition instead:
\begin{itemize}
    \item[(H3)] For each $t\geq 0$, the restriction of $P_t$ to $L^2(E,\mu)\cap \bb \sB(E)$ extends uniquely to a contraction operator on $L^2(E,\mu)$. The infinitesimal generator of $(P_t)_{t\geq 0}$ on $L^2(E,\mu)$ is given by
    \[
        \sL f:=\lim_{t\rightarrow 0}\frac{P_t f-f}{t},
    \]
    where the domain $\cD(\sL)$ of $\sL$ consists of all functions $f\in L^2(E,\mu)$ for which the above limit exists in the strong sense in $L^2(E,\mu)$. Furthermore, there exists a finite constant $K>0$ such that the bilinear form 
    \[
        \EE(f,g):=(f,-Lg)_\mu,\quad f,g\in \cD(\sL)
    \]
    satisfies 
    \begin{equation}\label{eq:56}
        |\EE_1(f,g)|\leq K\cdot \EE_1(f,f)^{1/2} \EE_1(g,g)^{1/2},\quad \forall f,g\in \cD(\sL),
    \end{equation}
    where $\EE_1(f,g):=\EE(f,g)+ (f,g)_\mu$ for $f,g\in \cD(\sL)$. 
\end{itemize}
Under this framework, let $\FF$ be the completion of $\cD(\sL)$ with respect to the norm $\|f\|_{\EE_1}:=\EE_1(f,f)^{1/2}$. Then $(\EE,\FF)$ is a semi-Dirichlet form on $L^2(E,\mu)$, as shown in \cite{F01}. 
When $X$ is transient, we use a slightly stronger version of (H3) (only for the case $\lambda=0$ as in Lemma~\ref{LM53}): 
\begin{itemize}
    \item[(H3')] All the assumptions in (H3) are retained, except that \eqref{eq:56} is replaced by the following condition:
    \[
        |\EE(f,g)|\leq K\cdot \EE(f,f)^{1/2} \EE(g,g)^{1/2},\quad \forall f,g\in \cD(\sL). 
    \]
\end{itemize}
In this case, let $\FF_\re$ denote the completion of $\cD(\sL)$ with respect to the norm $\|f\|_\EE:=\EE(f,f)^{1/2}$, and define $\FF:=\FF_\ee \cap L^2(E,\mu)$. Then $(\EE,\FF)$ is also a semi-Dirichlet form on $L^2(E,\mu)$ with $\FF_\ee$ being its extended Dirichlet space. (Under assumption (H3), we likewise use $\FF_\re$ to denote the extended Dirichlet space of $(\EE,\FF)$.)

\begin{rmk}
Clearly, (H3') implies (H3). 
If $X$ satisfies (H3), then  for any $\lambda>0$, its subprocess $X^{(\lambda)}$, with transition function $(\ee^{-\lambda t}P_t)_{t\geq 0}$, satisfies (H3'). In addition, the symmetry assumption (H1a) implies (H3) and (H3'). However,  under either (H3) or (H3'), $X$ need not admit a dual process as in (H2a). The assumptions (H2a) and (H3) together constitute the classic  framework of non-symmetric Dirichlet forms; see \cite{MR}.
\end{rmk}

According to \cite[Theorem~3.22]{F01}, if $X$ satisfies (H3) (or (H3') when $X$ is transient), then $(\EE,\FF)$ is quasi-regular on $L^2(E,\mu)$. In particular, every function $f\in \FF_\re$ admits an $\EE$-quasi-continuous $\mu$-a.e. version.  For conveniece, we always take $f\in \FF_\re$ to be such version. 

We are now in a position to state the analogue of Lemma~\ref{LM51} under the sector condition. For $\lambda\geq 0$, let $\mathscr{M}^{(\lambda)}$ denote the class of positive measures $\nu$ on $E$ such that $\nu U_\lambda$ is $\sigma$-finite and absolutely continuous with respect to $\mu$, and such that the Radon-Nikodym derivative $\dd (\nu U_\lambda)/\dd \mu$ lies in $\FF$ if $\lambda>0$, and in $\FF_\re$ if $\lambda=0$ and $X$ is transient. Write $\mathscr{M}:=\mathscr{M}^{(0)}$ for convenience. 

 \begin{lem}\label{LM53}
    Assume that the Borel right process $X$ satisifes (AC) and (H3'), and that $X$ is transient.  Let $B\in \sB^n(E)$ be a nearly Borel set such that $p_B\in \FF_\re$. Then there exists a unique measure $\nu_B\in \mathscr{M}$ such that $\FF_\re\subset L^1(E,\nu_B)$ and
    \begin{equation}\label{eq:57-2}
        \EE(g,p_B)=\int_E g(x)\nu_B(\dd x),\quad \forall g\in \FF_\re. 
    \end{equation}
    In particular, 
    \begin{equation}\label{eq:57}
        p_B(x)=\int_E u(x,y)\nu_B(\dd y),\quad \forall x\in E,
    \end{equation}
    and $\nu_B$ is a finite measure concentrated on $\overline{B}$ with $\nu_B(\overline{B})=\EE(p_B,p_B)$. 
 \end{lem}
 \begin{proof}
    The existence of $\nu_B$ has already been established in  \cite[Proposition~4.12]{F01}. It is also shown in \cite[Remark~4.19(b)]{F01} that \eqref{eq:57} holds for $\mu$-a.e. $x\in E$. Since both $p_B$ and $x\mapsto \int_E u(x,y)\nu_B(\dd y)$ are excessive functions with respect to $X$, it follows that \eqref{eq:57} holds for all $x\in E$. 

    We claim that $\nu_B(D^c)=0$, where $D:=B\cup B^r$. 
    Since $D\subset \overline{B}$, this claim will establish $\nu_B(\overline{B}^c)=0$. Indeed, $D$ is finely closed and $G:=D^c$ is finely open with respect to $X$.  
    Let $\FF^G:=\{g\in \FF:g=0, \EE\text{-q.e. on }D\}$ and let $\EE^G$ denotes  the restriction of $\EE$ to $\FF^G\times \FF^G$. By \cite[Theorem~5.10]{F01}, $(\EE^G,\FF^G)$ is quasi-regular on $L^2(E_G,\mu_{E_G})$, where $E_G:=\{x\in E:\bP^x[\sigma_D>0]=1\}$ and $\mu_{E_G}:=\mu|_{E_G}$ is the restriction of $\mu$ to $E_G$. On the other hand, note that $\sigma_B=\sigma_D$, $\bP^x$-a.s. for all $x\in E$, and $D^r=B^r$ (see \cite[(10.6)]{S88}); this implies that 
    \begin{equation}\label{eq:59}
        p_B\equiv p_D=P_D\mathbf{1}_E.
    \end{equation}
    By \cite[Theorem~4.3]{F01} (semipolar sets are $\mu$-polar) and \cite[Theorem~A.2.17~(ii)]{CF}, $B\setminus B^r$ is polar. We have $X_{\sigma_D}\in B^r=D^r$, $\bP^x$-a.s. on $\{\sigma_D<\infty\}$. It follows from \cite[Chapter I, (11.9)]{BG68} and \eqref{eq:59} that 
    \[
        p_B=P_D\mathbf{1}_E=P_DP_D\mathbf{1}_E=P_D p_B. 
    \]
 Applying \cite[Lemma~5.4(ii)]{F01} to $P_D p_B$ ($=p_B$), we obtain
    \[
        \EE(g,p_B)=0,\quad \forall g\in \FF^G. 
    \]
    It then follows from \eqref{eq:57-2} that
    \[
    \int_E g(x)\nu_B(\dd x)=0,\quad \forall g\in \FF^G. 
    \]
    Taking a function $g\in \pp\FF^G$ such that $g(x)>0$ for all $x\in G$, we can conclude that $\nu_B(G)=0$. (For example, let $(U^G_\beta)_{\beta>0}$ be the resolvent of the process $X^G$ obtained by killing $X$ upon leaving $G$; note that $X^G$ is properly associated with $(\EE^G,\FF^G)$. Taking a strictly positive function $h\in \sB(E_G)\cap L^2(E_G,\mu_{E_G})$ and setting $g:=U^G_1h$ give such a $g$.) 

    Finally, we show that $\nu_B(D)<\infty$. By \cite[Theorem~4.3]{F01}, $B\setminus B^r$ is $\mu$-polar, and it then follows from \cite[Theorem~4.1]{F01} that $\nu_B(B\setminus B^r)=0$.  Since $p_B\in \FF_\re\subset L^1(E,\nu_B)$ and $p_B=p_D\equiv 1$ on $B^r$, we have
    \[
        \nu_B(D)=\nu_B(B^r)=\int_E p_B(x)\nu_B(\dd x)=\EE(p_B,p_B)<\infty. 
    \]
    This completes the proof. 
 \end{proof}
 \begin{rmk}
 If $\mathbf{1}_E\in \FF_\re$, then $p_B\in \FF_\re$ for every $B\in \sB^n(E)$. This follows from the identity $p_B=p_{B\cup B^r}$ and \cite[Lemma~5.4(1)]{F01}.
 \end{rmk}

 By considering the subprocess $X^{(\lambda)}$, we can readily establish the integral representation of $p^{(\lambda)}_B$ for $\lambda>0$ under assumption (H3).

 \begin{corollary}\label{COR55}
     Assume that the Borel right process $X$ satisifes (AC) and (H3).  Let $\lambda>0$, and let $B\in \sB^n(E)$ be  such that $p^{(\lambda)}_B\in \FF$. Then there exists a unique measure $\nu^{(\lambda)}_B\in \mathscr{M}^{(\lambda)}$ such that $\FF\subset L^1(E,\nu^{(\lambda)}_B)$ and
    \[
        \EE_\lambda(g,p_B)=\int_E g(x)\nu^{(\lambda)}_B(\dd x),\quad \forall g\in \FF. 
    \]
    In particular, 
    \[
        p^{(\lambda)}_B(x)=\int_E u_\lambda(x,y)\nu^{(\lambda)}_B(\dd y),\quad \forall x\in E,
    \]
    and $\nu^{(\lambda)}_B$ is a finite measure concentrated on $\overline{B}$ with 
    \[
\nu^{(\lambda)}_B(\overline{B})=\EE_\lambda (p^{(\lambda)}_B,p^{(\lambda)}_B):=\EE(p^{(\lambda)}_B,p^{(\lambda)}_B)+\lambda\cdot (p^{(\lambda)}_B,p^{(\lambda)}_B)_\mu.
    \]
 \end{corollary}
\begin{proof}
  Let $\EE_\lambda(f,g):=\EE(f,g)+\lambda\cdot (f,g)_\mu$ for $f,g\in \FF$.  The subprocess $X^{(\lambda)}$ is associated with the semi-Dirichlet form $(\EE^{(\lambda)},\FF^{(\lambda)})=(\EE_\lambda, \FF)$, whose extended Dirichlet space is again $\FF^{(\lambda)}_\ee=\FF$. 
    As shown in \eqref{eq:54}, $p^{(\lambda)}_B$ coincides with the hitting probability of $B$ for the subprocess $X^{(\lambda)}$. Therefore, the desired conclusions follow immediately from  Lemma~\ref{LM53}. 
\end{proof} 
\begin{rmk}
    If $\mathbf{1}_E\in \FF$ (for example, if $X$ is conservative in the sense that $P_t\mathbf{1}_E=\mathbf{1}_E$ for all $t>0$ and $\mu$ is finite), then $p^{(\lambda)}_B\in \FF$ for any $\lambda>0$ and any $B\in \sB^n(E)$. 
\end{rmk}

With the integral representation established in Lemma~\ref{LM53} and Corollary~\ref{COR55}, we can obtain the following result. 

\begin{thm}\label{THM:5.8}
    Let $X$ be a Borel right process on $(E,d)$ that satisfies (AC) and the following short-time heat kernel lower bound at a fixed point $x\in E$: For some $T>0, \alpha>0$ and $\beta>1$, there exist constants $C_1,C_2>0$ such that
\begin{equation}\label{eq:511-2}
p(t,x,y)\ge \frac{C_1}{t^{\alpha/\beta}}\exp\left\{-C_2\left(\frac{d(x,y)}{t^{1/\beta}}\right)^{\beta/(\beta -1)}\right\}, \quad \forall t\in (0, T], \, y\in E.
\end{equation}
Let $B\in \sB^n(E)$ be a nearly Borel set.
\begin{itemize}
    \item[(1)] If $X$ satsifies (H3) and $p^{(\lambda)}_B\in \FF$ for some $\lambda>0$, then for any $\delta,\rho>0$,
 \[
        \bP^x[\sigma_B\leq t]\geq \frac{\nu^{(\lambda)}_{B_{x,\delta}}(\overline{B}_{x,\delta})}{\re^{\lambda T}}\cdot \frac{C_1 C_3}{ t^{(\alpha-\beta)/\beta}}\exp\left\{-(1+\rho)C_2\left( \frac{\widetilde{d}(x,B)+\delta}{t^{1/\beta}}\right)^{\beta/(\beta-1)}\right\},\quad 0<t\leq T,
    \]
    where  $B_{x,\delta}:=\{y\in B:d(y,x)\leq \widetilde{d}(x,B)+\delta\}$, $\nu^{(\lambda)}_{B_{x,\delta}}$ is the measure obtained in Corollary~\ref{COR55}, $C_1$ and $C_2$ are as in \eqref{eq:511-2}, and $C_3=(\beta-1)\rho\cdot \min\{(1+\rho)^{\alpha-\beta-\alpha/\beta}, 1\}>0$ is a constant depending on $\alpha, \beta,\rho$.
\item[(2)] If $X$ is transient and satisfies (H3'), and $p_B\in \FF_\re$, then the same conclusion as in the first statement holds with $\lambda=0$ and $\nu^{(0)}_{B_{x,\delta}}=\nu_{B_{x,\delta}}$, where $\nu_{B_{x,\delta}}$ is the measure obtained in Lemma~\ref{LM53}. 
\end{itemize}
\end{thm}
\begin{proof}
We consider  only the transient case, under assumption (H3'). 
    The proof is identical to that of  Theorem~\ref{THM52}, except for showing that $p_{B_{x,\delta}}\in \FF_\ee$. To this end, let $D_\delta:=B_{x,\delta}\cup B^r_{x,\delta}$, where $B^r_{x,\delta}$ denotes the set of all regular points for $B_{x,\delta}$. By \cite[(10.6)]{S88} and the fact that $B_{x,\delta}\subset B$, we have $D^r_\delta=B^r_{x,\delta}\subset B^r=D^r$, where $D:=B\cup B^r$. In addition, $p_{B_{x,\delta}}=p_{D_\delta}$ and $p_B=p_D$.  According to \cite[Theorem~4.3]{F01} and \cite[Theorem~A.2.17~(ii)]{CF}, $B_{x,\delta}\setminus B^r_{x,\delta}$ is polar. Therefore, $X_{\sigma_{D_{\delta}}}\in B^r_{x,r}\subset B^r=D^r$, $\bP^y$-a.s. on $\{\sigma_{D_{\delta}}<\infty\}$ for all $y\in E$. Using \cite[Chapter I, (11.9)]{BG68}, we obtain 
    \begin{equation}\label{eq:511}
p_{B_{x,\delta}}=p_{D_\delta}=P_{D_\delta}\mathbf{1}_E=P_{D_\delta}P_D\mathbf{1}_E=P_{D_\delta}p_B. 
    \end{equation}
    Then $p_{B_{x,\delta}}\in \FF_\re$ follows from \cite[Lemma~5.4~(i)]{F01}. This completes the proof.
\end{proof}

We close this subsection with a useful sufficient condition for $p_B\in \FF_\re$, and for $p^{(\lambda)}_B\in \FF$ when $\lambda>0$.  The semi-Dirichlet form $(\EE,\FF)$ is said to be \emph{regular} if $(E,d)$ is a locally compact separable metric space, $\mu$ is a Radon measure on $E$, and $\FF\cap C_c(E)$ is dense in $\FF$ with respect to the $\EE_1$-norm, and dense in $C_c(E)$ with respect to the uniform norm; see, e.g., \cite{O13}. 

\begin{lem}
Assume that the Borel right process $X$ is associated with a regular semi-Dirichlet form $(\EE,\FF)$ on $L^2(E,\mu)$. Let  $B\in \sB^n(E)$ be such that its closure $\overline{B}$ is compact in $E$. Then $p_B\in \FF_\re$ if $X$ is transient, and $p^{(\lambda)}_B\in \FF$ for any $\lambda>0$. 
\end{lem}
\begin{proof}
    We consider only the transient case, and show that $p_B\in \FF_\re$. Let $D:=B\cup B^r\subset \overline{B}$. As in \eqref{eq:511}, we can obtain that
    \[
p_B=p_D=P_D\mathbf{1}_E=P_DP_{\overline{B}}\mathbf{1}_E=P_Dp_{\overline{B}}. 
    \]
    By \cite[Lemma~5.4(i)]{F01}, it suffices to show that $p_{\overline{B}}\in \FF_\re$. Indeed, using the regularity of $(\EE,\FF)$, we can find a function $f\in \FF\cap C_c(E)$ such that $f>1$ on $\overline{B}$. Setting $g:=f^+\wedge 1\in \FF$, where $f^+:=\max(f,0)$, the definition of semi-Dirichlet form (see \cite[(2.7)]{F01}) gives $g\in \FF$. It then follows from \cite[Lemma~5.4(i)]{F01} that $P_{\overline{B}}g\in \FF_\re$. Note that $X_{\sigma_{\overline{B}}}\in \overline{B}$, $\bP^x$-a.s. for all $x\in E$, and $g|_{\overline{B}}\equiv 1$. Therefore,
    \[
        P_{\overline{B}}g(x)=\bE^x\left[g(X_{\sigma_{\overline{B}}}); \sigma_{\overline{B}}<\infty\right]=\bP^x[\sigma_{\overline{B}}<\infty]=p_{\overline{B}}(x),\quad \forall x\in E. 
    \]
    This shows $p_{\overline{B}}\in \FF_\re$, which completes the proof. 
    \end{proof}

\section{Small-time asymptotics of hitting time distributions}\label{SEC6}

We work under the same standing assumptions as in the previous sections. Namely, $(E,d)$ is a separable metric space whose topology induced by $d$ is Lusin, $\mu$ is a $\sigma$-finite measure on $E$ with full support, and $X$ is a Borel right process on $E$ satisfying (AC). 
As an application of the estimates for the hitting time distribution, we investigate the asymptotic behavior of $\bP^x[\sigma_B\le t]$ as $t\downarrow0$. 

By Blumenthal $0$-$1$ law, for any $B\in \sB^n(E)$,
\[
    \lim_{t\downarrow 0}\bP^x[\sigma_B\leq t]=\bP^x[\sigma_B=0]=\left\lbrace
    \begin{aligned}
    &1,\quad x\in B^r,\\
    &0,\quad x\notin B^r. 
    \end{aligned}
    \right.
\]
In particular, if $x\in  B^r$, then $\bP^x[\sigma_B\leq t]=1$ for any $t>0$ and
\[
    \lim_{t\downarrow 0}\log \bP^x[\sigma_B\leq t]=0. 
\]
Otherwise if $x\notin B^r$, then $\lim_{t\downarrow 0}\log \bP^x[\sigma_B\leq t]=-\infty$.

We now proceed to the first of two more precise asymptotic results, which follows from the upper bound estimate established in Theorem~\ref{THM:5.1}.

\begin{thm}\label{THM61-2}
   Assume that  the volume growth condition \eqref{eq:41} holds  with some constant $\alpha>0$, and the transition density function of $X$ satisfies 
\[
p(t,y,z)\leq  \frac{C_1}{t^{\alpha/\beta}}\exp\left\{-C_2\left(\frac{d(y,z)}{t^{1/\beta}}\right)^{\beta/(\beta -1)}\right\}, \quad \forall  0<t\leq T, \, y,z\in E,
\]
for some $\beta>1$, $T>0$, $C_1>0$ and $C_2>0$, then for any $x\in E$ and $B\in \sB^n(E)$,
\begin{equation}\label{eq:64-2}
        \liminf_{t\downarrow 0}\frac{-t^{1/(\beta-1)}\cdot \left(\log \bP^x[\sigma_B\leq t]\right)}{C_2}\geq  \widetilde{d}(x,B)^{\beta/(\beta-1)}.
\end{equation}
\end{thm}
\begin{proof}
    We only consider $B\neq \emptyset$, since the statement holds trivially when $B=\emptyset$. Note that the upper bound estimate of $\bP^x[\sigma_B\leq t]$ established in Theorem~\ref{THM:5.1}  only relies on the the upper bound of $p(s,x,y)$ for all $0<s\leq t$. By using \eqref{eq:33}, we can easily obtain \eqref{eq:64-2}. 
\end{proof}

Another asymptotic result follows from the lower estimate of $\bP^x[\sigma_B\leq t]$. The crucial fact is that $B_{x,\delta}=\{y\in B:d(y,x)\leq \widetilde{d}(x,B)+\delta\}$ is not polar whenever $B$ is not polar. 

\begin{thm}\label{THM61}
   Assume that every $d$-bounded closed subset of $E$ is compact, and that $X$ satisfies one of the following sets of assumptions:
\begin{itemize}
    \item[{\rm(i)}] the symmetry assumptions (H1a) and (H1b);
    \item[{\rm(ii)}] the duality assumptions (H2a) and (H2b);
    \item[{\rm(iii)}] the sector condition (H3), and the associated semi-Dirichlet form $(\EE,\FF)$ of $X$ is regular.
\end{itemize}
 If the transition density function of $X$ satisfies the following short-time lower bound estimate at a fixed point $x\in E$:
\[
p(t,x,y)\geq  \frac{C_1}{t^{\alpha/\beta}}\exp\left\{-C_2\left(\frac{d(x,y)}{t^{1/\beta}}\right)^{\beta/(\beta -1)}\right\}, \quad \forall  0<t\leq T, \, y\in E,
\]
for some $\alpha>0$, $\beta>1$, $C_1>0$, $C_2>0$ and $T>0$, then for any $B\in \sB^n(E)$, 
\[
        \limsup_{t\downarrow 0}\frac{-t^{1/(\beta-1)}\cdot \left(\log \bP^x[\sigma_B\leq t]\right)}{C_2}\leq  \widetilde{d}(x,B)^{\beta/(\beta-1)}.
\]
\end{thm}
\begin{proof}
We only consider the case that $B$ is not polar. 
According to Lemma~\ref{LM41}, Lemma~\ref{LM51} and Corollary~\ref{COR55}, for any $\delta>0$, there exists a finite measure $\nu^{(1)}_{B_{x,\delta}}$ such that 
\[
    p^{(1)}_{B_{x,\delta}}(x)=\int_E u_1(x,y)\nu^{(1)}_{B_{x,\delta}}(\dd y). 
\]
Since $B_{x,\delta}$ is not polar by Lemma~\ref{LM25}, it follows that $\nu^{(1)}_{B_{x,\delta}}(\overline{B}_{x,\delta})=\nu^{(1)}_{B_{x,\delta}}(E)>0$. 
In particular, 
taking $\rho=\delta>0$ and $\lambda=1$ in \eqref{eq:45} for the symmetry case (and the corresponding estimates in Theorems~\ref{THM52} and \ref{THM:5.8} for the other two cases), we obtain that for any $0<t\leq T$,
   \[
      - \log \bP^x[\sigma_B\leq t]\leq -\log\left(C_1C_3\cdot \nu^{(1)}_{B_{x,\delta}}(\overline{B}_{x,\delta})\right)+ T+\frac{\alpha-\beta}{\beta}\log t+(1+\delta)C_2\left(\frac{\widetilde{d}(x,B)+\delta}{t^{1/\beta}}\right)^{\beta/(\beta-1)}.
   \]
   (Note that $\overline{B}_{x,\delta}$ is $d$-bounded closed and hence, compact in $E$. Therefore, $B_{x,\delta}$ satisfies all the conditions in Theorems \ref{THM52} and \ref{THM:5.8}.)
   It follows that
   \begin{equation}\label{eq:62-2}
       \limsup_{t\downarrow 0}\frac{-t^{1/(\beta-1)}\cdot \left(\log \bP^x[\sigma_B\leq t]\right)}{C_2}\leq (1+\delta)(\widetilde{d}(x,B)+\delta)^{\beta/(\beta-1)}. 
   \end{equation}
   Then letting $\delta\downarrow 0$ gives
   \[
       \limsup_{t\downarrow 0}\frac{-t^{1/(\beta-1)}\cdot \left(\log \bP^x[\sigma_B\leq t]\right)}{C_2}\leq \widetilde{d}(x,B)^{\beta/(\beta-1)}.
   \]
   This completes the proof. 
\end{proof}
\begin{rmk}
\begin{itemize}
 \item[(1)] If the closure $\overline{B}$ of $B$ is compact in $E$, then the assumption that every $d$-bounded closed subset of $E$ is compact can be removed. 
\item[(2)] Under the same assumptions as in Theorem~\ref{THM61}, let $B=G$ be an open set and consider a point $x\in\partial G$. In particular, $\widetilde d(x,G)=0$. If $x\notin G^r$, namely, $\lim_{t\downarrow 0}\bP^x[\sigma_G\leq t]=\bP^x[\sigma_G=0]=0$, then Theorem~\ref{THM61} immediately yields
\[
    \lim_{t\downarrow 0}\left(-t^{1/(\beta-1)}\cdot \left(\log \bP^x[\sigma_G\leq t]\right)\right)=0.
\]
A concrete example illustrating this phenomenon is provided in Section~\ref{SEC74}.
\end{itemize}
\end{rmk}

The two theorems above together imply the following result. 

\begin{corollary}\label{COR63}
    Assume that the assumptions of Theorems~\ref{THM61-2} and \ref{THM61} are both satisfied. If the transition density function of $X$ satisfies
      \begin{equation}\label{eq:63-2}
    p(t,y,z)\asymp \frac{1}{t^{\alpha/\beta}}\exp \left\{-C_2\left(\frac{d(y,z)}{t^{1/\beta}}\right)^{\beta/(\beta-1)} \right\},\quad 0<t\leq T, y,z\in E
 \end{equation}
 for some $\alpha>0,\beta>1, C_2>0$ and $T>0$, then
 for any $x\in E$ and $B\in \sB^n(E)$,
    \begin{equation}\label{eq:68}
        \lim_{t\downarrow 0}\frac{-t^{1/(\beta-1)}\cdot \left(\log \bP^x[\sigma_B\leq t]\right)}{C_2}= \widetilde d(x,B)^{\beta/(\beta-1)}.
    \end{equation}
\end{corollary}
\begin{rmk}
  From the two-sided sub-Gaussian heat kernel estimates \eqref{eq:63-2}, 
one can readily deduce the following Varadhan-type small-time asymptotic (see \cite{Varadhan, V67}):
\begin{equation}\label{eq:65}
    \lim_{t\downarrow 0} \frac{-t^{1/(\beta-1)}\cdot \log p(t,y,z)}{C_2}=d(y,z)^{\beta/(\beta-1)},\quad y,z\in E. 
\end{equation}
A question of particular interest, which remains open, is whether the asymptotic behavior of the hitting time distribution in \eqref{eq:68} can be derived directly from such a Varadhan-type asymptotic \eqref{eq:65}. We hope to address this problem in future work.
\end{rmk}

The following example satisfies the asymptotic relation \eqref{eq:68} obtained in Corollary~\ref{COR63}. It also illustrates why $\widetilde{d}(x,B)$, rather than $d(x,B)$, is used in the estimates of $\bP^x[\sigma_B\leq t]$.

\begin{exe}\label{EXA63}
 Let $X$ be the Brownian motion on $\bR^n$ equipped with the Euclidean metric $d(x,y):=|x-y|$, where $n\geq 2$. Then all the assumptions of Corollary~\ref{COR63} are satisfied with $\alpha=n,\beta=2$  and $C_2=1/2$. 
 
 Let $\mathbf{o}$ denote the origin  of $\bR^n$, $\mathbf{1}:=(1,0,\cdots,0)\in \bR^n$ and define
  \[
      B_1:=\{x\in \bR^n: |x|\geq 2\},\quad B_2:=B_1\cup \{\mathbf{1}\}. 
  \]
  Note that the singleton $\{\mathbf{1}\}$ is polar. By Lemma~\ref{LM24}, $\widetilde{d}(x,B_2)=\widetilde{d}(x,B_1)=d(x,B_1)$. 
  Consequently,
  \[
      \lim_{t\downarrow 0}\left(-t \log \bP^\mathbf{o}[\sigma_{B_2}\leq t]\right)=\lim_{t\downarrow 0}\left(-t \log \bP^\mathbf{o}[\sigma_{B_1}\leq t]\right)=\frac{1}{2}d(x,B_1)^2=2, 
  \]
  where the second equality follows from Corollary~\ref{COR63}. This agrees with the well-known small-time asymptotics for Brownian hitting times.
\end{exe}

\section{Examples and applications}\label{Sec-applications}

In this section, we apply the results in previous sections to a few specific types of Markov processes. 

\subsection{Heat kernels with two-sided (sub-)Gaussian estimates}\label{SEC71}

The transition density function of $X$ is said to satisfy the two-sided short-time sub-Gaussian estimates if, for some $T>0$ and positive constants $c_1, c_2, C_1, C_2$,
\begin{equation}\label{eq:71}
    \frac{C_1}{t^{\alpha/\beta}}\exp \left\{-C_2\left(\frac{d(y,z)}{t^{1/\beta}}\right)^{\beta/(\beta-1)} \right\}\leq p(t,y,z)\leq \frac{c_1}{t^{\alpha/\beta}}\exp \left\{-c_2\left(\frac{d(y,z)}{t^{1/\beta}}\right)^{\beta/(\beta-1)} \right\}
\end{equation}
for all $0<t\leq T$ and $y,z\in E$, where $c_2$ and $C_2$ are not necessarily equal. When $\beta=2$, this reduces to the classical Gaussian estimates, with $\alpha$ corresponding to the volume growth exponent (which is typically the dimension of the state space in the Euclidean setting).
Suppose now that the transition density satisfies the above two-sided short-time sub-Gaussian estimates. Then the lower and upper bounds can be applied separately in Theorems~\ref{THM61-2} and~\ref{THM61}, respectively, whenever the additional assumptions required in the corresponding theorems are satisfied. Consequently, for any $x\in E$ and $B\in \sB^n(E)$,
\[
\begin{aligned}
c_2\cdot \widetilde{d}(x,B)^{\beta/(\beta-1)}&\leq 
\liminf_{t\downarrow 0}\left(-t^{1/(\beta-1)}\cdot \left(\log \bP^x[\sigma_B\leq t]\right)\right) \\
&\leq 
 \limsup_{t\downarrow 0}\left(-t^{1/(\beta-1)}\cdot \left(\log \bP^x[\sigma_B\leq t]\right)\right)\leq  C_2\cdot \widetilde{d}(x,B)^{\beta/(\beta-1)}.
\end{aligned}\]

There is a vast literature on two-sided (sub-)Gaussian heat kernel estimates. Most of the known examples are established in the symmetric setting and therefore correspond to the first case of Theorem~\ref{THM61}. Typical examples include symmetric uniformly elliptic diffusion processes on Euclidean domains \cite{Aro67, Aro68, Dav89,  FS86}, symmetric diffusion processes on complete Riemannian manifolds satisfying the volume doubling property and the Poincar\'e inequality \cite{Gri94, Sal92, Sal10}, and Brownian motions on fractals such as the Sierpinski gasket and Sierpinski carpet \cite{Barlow1,BaPe88,FHK,Ki01,Kum93}. 
 More generally, two-sided sub-Gaussian estimates have been obtained for symmetric diffusions associated with strongly local regular Dirichlet forms on a broad class of metric measure spaces; see \cite{BBK06,GT12}.

 In contrast to the symmetric examples discussed above, two-sided Gaussian heat kernel estimates are also available for certain non-symmetric diffusion processes. A typical example is the diffusion process on $\mathbb{R}^n$ with generator
\[
\frac{1}{2}\Delta+b\cdot\nabla .
\]
When the drift term $b$ belongs to the Kato class (in particular, when $b$ is bounded and measurable), it was proved in \cite{KS06} that the corresponding transition density satisfies two-sided Gaussian heat kernel estimates. Moreover, these processes fit into the non-symmetric framework considered in Theorem~\ref{THM61}. More precisely, they satisfy both the duality assumption (H2a) and the sector condition (H3), corresponding to the second and third settings, respectively, in Theorem~\ref{THM61}. The construction of the corresponding regular non-symmetric Dirichlet forms can be found in \cite[Theorem~3.5]{ChenZhao}.

We emphasize that all the examples discussed above correspond to diffusions associated with local or strongly local (semi-)Dirichlet forms. Nevertheless, the framework of Theorems~\ref{THM61-2} and \ref{THM61} is not restricted to this setting and also applies to processes with jumps, including those associated with non-local (semi-)Dirichlet forms.

\subsection{Reflected Brownian motion on convex domains}

We next give an example in which the constants in the two-sided Gaussian estimates can be chosen to be identical, namely, $c_2=C_2$. This is in contrast to the previous examples, where the upper and lower heat kernel estimates generally involve different exponential constants.

We briefly recall some facts about reflected Brownian motion on possibly nonsmooth domains; see, for example, \cite{BurdzyChen}. Let $D\subset\mathbb{R}^n$ be a domain and consider the bilinear form
\[
\mathcal{E}(f,g)=\frac12\int_D \nabla f(x)\cdot\nabla g(x)\,\dd x
\]
on $L^2(D)$, with domain
\[
\mathcal{D}(\mathcal{E})=W^{1,2}(D).
\]
Assume that
\begin{equation}\label{RBM-regular-DF-condition}
C^1_c(\overline{D}) \text{ is dense in } (W^{1,2}(D),\, \|\cdot\|_{1,2}),
 \end{equation}
where $\|\cdot\|_{1,2}:= \|\cdot\|_{2}+\|\nabla \cdot\|_{2}$ and $\|\cdot\|_2$ denotes the norm on $L^2(D)$. Then $(\mathcal{E},\mathcal{D}(\mathcal{E}))$ is a regular Dirichlet form on $L^2(\overline{D})$ and hence admits an associated Hunt process, which is the reflected Brownian motion on $\overline D$.

The above regularity condition \eqref{RBM-regular-DF-condition} is satisfied, in particular, when $D$ is a $W^{1,2}$-extension domain. Since Lipschitz domains are uniform domains and uniform domains are $W^{1,2}$-extension domains (see \cite{PJones}), this framework covers a large class of possibly nonsmooth domains. 

For such domains, the Nash inequality
\begin{equation*}
\|f\|{}_{2}^{1+\frac{2}{n}}\le C\cdot \|\nabla f\|{}_{2} \|f\|{}_{1}^{2/n}, \quad \text{for any } f \in L^1(D) \cap W^{1,2}(D)
\end{equation*}
holds, where $\|\cdot\|_1$ denotes the norm on $L^1(D)$. Combined with Davies’ method, this yields the following Gaussian upper bound for the heat kernel of reflected Brownian motion:
\[
p(t,x,y)\leq \frac{c_1}{t^{n/2}}
\exp\left\{-\frac{d(x,y)^2}{2t}\right\},\qquad t>0,
\]
where $d(\cdot,\cdot)$ denotes the geodesic distance on $\overline D$ induced by the Euclidean metric.

For convex domains, a matching Gaussian lower bound is available. More precisely, Pascu \cite{Pascu} proved that the heat kernel of reflected Brownian motion on convex Euclidean domains satisfies
\[
p(t,x,y)\ge
\frac1{(2\pi t)^{n/2}}
\exp\left\{-\frac{|x-y|^2}{2t}\right\},
\qquad t>0.
\]
Since the geodesic distance on a convex domain coincides with the Euclidean distance, we obtain the two-sided Gaussian estimate
\[
\frac{1}{(2\pi t)^{n/2}}
\exp\left\{-\frac{|x-y|^2}{2t}\right\}
\le p(t,x,y)
\le
\frac{c_1}{t^{n/2}}
\exp\left\{-\frac{|x-y|^2}{2t}\right\},
\qquad t>0.
\]
In particular, the exponential constants in the upper and lower estimates are identical in this case.

As an immediate consequence of Corollary~\ref{COR63}, we obtain the following result.

\begin{thm}\label{THM:6.1}
Let $D\subset \IR^n$ be a convex domain that satisfies \eqref{RBM-regular-DF-condition}, and let  $d$ denote the Euclidean distance. Let $(X_t)_{t\ge 0}$ be the reflected Brownian motion on $\overline{D}$. For any   $x\in E$ and nearly Borel set $B\subset \overline{D}$,
\begin{equation*}
\lim_{t\downarrow 0} \left(-t\cdot \log \IP^x[\sigma_B\le t]\right)=\frac{\widetilde{d}(x, B)^2}{2},
\end{equation*}
where $\widetilde{d}(x,B)$ is defined as \eqref{eq:63}.
\end{thm}

\subsection{Brownian motion on Riemannian manifolds}

The previous example provides a concrete example of Corollary~\ref{COR63}, where the exponential constants in the two-sided Gaussian estimates coincide. The following example shows that the same asymptotic relation remains valid when the two constants are not identical but can be chosen arbitrarily close.

Let $M$ be an $n$-dimensional complete Riemannian manifold with non-negative Ricci curvature. Let $\Delta_M$ denote the Laplace-Beltrami operator on $M$, and let $p_M(t,x,y)$ be the heat kernel associated with the heat semigroup generated by $\frac{1}{2}\Delta_M$. In their celebrated work \cite{LiYau}, Li and Yau proved that, for any $\epsilon>0$, there exist constants $C_1,c_1>0$, depending on $\epsilon$, such that
\begin{equation}\label{eq:72}
\frac{C_1}{\operatorname{Vol}(B(x,\sqrt{t}))}
\exp\left\{-\frac{d(x,y)^2}{(2-\epsilon)t}\right\}
\le p_M(t,x,y)
\le
\frac{c_1}{\operatorname{Vol}(B(x,\sqrt{t}))}
\exp\left\{-\frac{d(x,y)^2}{(2+\epsilon)t}\right\},
\end{equation}
where $\operatorname{Vol}$ denotes the volume measure on $M$ and $d$ is the Riemannian distrance, for all $x,y\in M$ and $t>0$.
Note that $\mathrm{Vol}$ satisfies the volume growth condition \eqref{eq:41} with $\mu=\mathrm{Vol}$ and $\alpha=n$. We also point out that \eqref{eq:72} is slightly different from the two-sided sub-Gaussian estimate \eqref{eq:71}, since the coefficients in \eqref{eq:72} involve the volume term $\operatorname{Vol}(B(x,\sqrt{t}))$, rather than the power $t^{n/2}$.

\begin{thm}\label{THM:6.2}
Let $M$ be an $n$-dimensional complete Riemannian manifold with non-negative Ricci curvature and $\Delta _M$ be the Laplace-Beltrami operator on it. 
 Let $X=(X_t)_{t\ge 0}$ be the Brownian motion on $M$ associated with the generator $\frac{1}{2}\Delta _M$. Then for any $x\in M$ and nearly Borel set $B\subset M$, 
\[
    \lim_{t\downarrow 0}\left(-t \cdot \log \bP^x[\sigma_B\leq t] \right)=\frac{1}{2}\widetilde{d}(x,B)^2,
\]
where $\widetilde{d}(x,B)$ is defined as \eqref{eq:63}. 
\end{thm}
\begin{proof}
    Note that the Brownian motion on $M$ satisfies the symmetry assumptions (H1a) and (H1b); see, e.g., \cite[\S2.2.5]{CF}. Since
 the volume measure satisfies the volume growth condition \eqref{eq:41} with $\alpha=n$, the lower bound in \eqref{eq:72} implies that
 \begin{equation}\label{eq:74-2}
     p_M(t,x,y)\geq \frac{C_1}{C_0 t^{n/2}}\exp\left\{-\frac{d(x,y)^2}{(2-\epsilon)t}\right\},\quad \forall t>0,x,y\in M.
 \end{equation}
Applying Theorem~\ref{THM61-2} with the above heat kernel lower bound, we obtain that, for the same $\epsilon>0$,
    \[
    \limsup_{t\downarrow 0}\left(-t \cdot \log \bP^x[\sigma_B\leq t] \right)\leq \frac{1}{2-\epsilon}\widetilde{d}(x,B)^2.
\]
Letting $\epsilon\downarrow 0$ gives 
\begin{equation}\label{eq:75-3}
    \limsup_{t\downarrow 0}\left(-t \cdot \log \bP^x[\sigma_B\leq t] \right)\leq \frac{1}{2}\widetilde{d}(x,B)^2.
\end{equation}

To prove the converse statement, we only need to make a slight modification to the proof of Theorem~\ref{THM:5.1}. Fix $\delta>0$ and $\epsilon>0$. The arguments before and in the first step remain valid, and in particular, \eqref{eq2.1} continues to hold with $p_M(t,x,y)$ in place of $p(t,x,y)$. Moreover, by the upper bound estimate
\[
    p_M(t,x,y)\leq \frac{c_1 t^{n/2}}{\mathrm{Vol}(B(x,\sqrt{t}))}\cdot \frac{1}{t^{n/2}}\exp\left\{-\frac{d(x,y)^2}{(2+\epsilon)t}\right\}
\]
and the argument in the second step, we obtain that, for any $\lambda>0$, $t>0$ and $x\in E$,
\begin{equation}\label{eq:75-2}
    \int_E p_M(t,x,y)\re^{\lambda d(x,y)} \mathrm{Vol}(\dd y)\leq \frac{c_1t^{n/2}}{\mathrm{Vol}(B(x,\sqrt{t}))}C_0C_\delta \exp \left\{ \frac{(2+\epsilon) t}{4(1-\delta/2)}\lambda^2\right\},
\end{equation}
where $C_\delta$ is the constant given by \eqref{eq:47} with $C_2=1/(2+\epsilon)$, $\alpha=n$ and $\beta=2$. Next, let $q(t,x,y)$ denote the lower bound  appearing on the right-hand side of \eqref{eq:74-2}. Applying the argument in the third step to this $q(t,x,y)$, with $C_1/C_0$ and $1/(2-\epsilon)$ playing the roles of the corresponding constants in the sub-Gaussian bound, we can choose $\eta>0$, depending only on $n$, $C_0, C_1$ and $\epsilon$, such that 
\begin{equation}\label{eq:76}
    \inf_{0<s\leq t,z\in B(x,r)^c}\int_E p_M(s,z,y)\re^{-\lambda d(z,y)}\mathrm{Vol}(\dd y)\geq \frac{1}{2}\re^{-\lambda \eta t^{1/2}}.
\end{equation}
Substituting \eqref{eq:75-2} and \eqref{eq:76} into \eqref{eq2.1}, we obtain 
\[
    \bP^x[\sigma<t]\leq \frac{c_1 t^{n/2}}{\mathrm{Vol}(B(x,\sqrt{t}))}\cdot 2C_0C_\delta \exp\left\{(\eta t^{1/2}-r)\cdot \lambda+\frac{(2+\epsilon) t}{4(1-\delta/2)}\lambda^2 \right\}.
\]
Using the same argument as in the fourth step to estimate the exponential term, we obtain constants $C^{(1)}_3>0$  and $C^{(2)}_3>0$, as in  \eqref{eq:412} and \eqref{eq:413}, respectively, depending on $n$, $C_0, C_1, \epsilon, \delta$, such that
\[
    \bP^x[\sigma<t]\leq  \max\left\{C^{(3)}_1,\frac{c_1 t^{n/2}}{\mathrm{Vol}(B(x,\sqrt{t}))}\cdot 2C_0C_\delta \re^{C^{(3)}_2/(2+\epsilon)}\right\} \exp\left\{-\frac{1-\delta}{2+\epsilon}\cdot \frac{r^2}{t}\right\}.
\]
Combining this estimate with the argument preceding the first step, we conclude that for any $t>0$, $x\in E$ and $B\in \sB^n(E)$, 
\begin{equation}\label{eq:77}
    \bP^x[\sigma_B\leq t]\leq  \max\left\{C^{(3)}_1,\frac{c_1 t^{n/2}}{\mathrm{Vol}(B(x,\sqrt{t}))}\cdot 2C_0C_\delta \re^{C^{(3)}_2/(2+\epsilon)}\right\} \exp\left\{-\frac{1-\delta}{2+\epsilon}\cdot \frac{\widetilde{d}(x,B)^2}{t}\right\},
\end{equation}
where $\delta>0$ and $\epsilon>0$ are arbitrary. 

Fix $x,B,\delta$ and $\epsilon$.  It follows from \eqref{eq:77} that
\begin{equation}\label{eq:78}
    \liminf_{t\downarrow 0}\left(-t\cdot \log \bP^x[\sigma_B\leq t]\right)\geq \liminf_{t\downarrow 0}\left(-t \cdot \log C_{x,t}\right)+\frac{1-\delta}{2+\epsilon} \widetilde{d}(x,B)^2,
\end{equation}
where $C_{x,t}$ denotes the  coefficient appearing on the right-hand side of \eqref{eq:77}.  A direct calculation gives, for $t>0$,
\[
    t\cdot \log C_{x,t}=  \max \left\{t\cdot \log C^{(3)}_1,t\cdot \log\left( 2c_1C_0C_\delta \re^{C^{(3)}_2/(2+\epsilon)}\right)+\frac{n}{2}t\log t-t\log \left( \mathrm{Vol}(B(x,\sqrt{t}))\right) \right\}.
\]
By the Bishop–Gromov inequality, there exists a constant $c_0>0$, possibly depending on  $x$, such that
\[
    \mathrm{Vol}(B(x,s))\geq c_0 s^n,\quad \forall 0<s<1. 
\]
This, together with the volume growth condition \eqref{eq:41}, yields $\lim_{t\downarrow 0}t\log \left( \mathrm{Vol}(B(x,\sqrt{t}))\right)=0$. Consequently,
\[
    \liminf_{t\downarrow 0}(-t\cdot \log C_{x,t})=\lim_{t\downarrow 0}(-t\cdot \log C_{x,t})=0. 
\]
It then follows from \eqref{eq:78} that
\[
    \liminf_{t\downarrow 0}(-t\cdot \log \bP^x[\sigma_B\leq t]\geq \frac{1-\delta}{2+\epsilon} \widetilde{d}(x,B)^2.
\]
Letting $\delta,\epsilon \downarrow 0$, we obtain 
\[
    \liminf_{t\downarrow 0}(-t\cdot \log \bP^x[\sigma_B\leq t]\geq \frac{1}{2} \widetilde{d}(x,B)^2.
\]
Combining this with the estimate \eqref{eq:75-3}  completes the proof.
\end{proof}

\subsection{Thorn-like sets for Brownian motion}\label{SEC74}

Consider the Brownian motion on $\bR^n$ with $n\geq 3$. 
Let $T_h$ denote the thorn
\[
    T_h:=\left\{(x_1,\cdots,x_n): x_1\geq 0\text{ and } x_2^2+\cdots+x^2_n\leq h^2(x_1)\right\} \subset \bR^n,
\]
where $h(r), r\geq 0$, is a continuous function such that $h(r)>h(0)=0$ for $r>0$, $h(r)/r$ is nondecreasing for  sufficiently small $r$, and $h(r)/r$ is nonincreasing for sufficiently large $r$. 
By the celebrated Wiener criterion, the origin $\mathbf{o}$ of $\bR^n$ is irregular for $T_h$, i.e., $\bP^\mathbf{o}[\sigma_{T_h}=0]=0$, if and only if
\begin{equation}\label{eq:74}
    \int_0^1 \left|\log \frac{h(r)}{r} \right|^{-1}\frac{\dd r}{r}<\infty\qquad \text{for }n=3,
\end{equation}
and
\begin{equation}\label{eq:73}
    \int_0^1\left(\frac{h(r)}{r}\right)^{n-3}\frac{\dd r}{r}<\infty\qquad \text{for }n>3;
\end{equation}
see, e.g., \cite[Proposition~3.5]{PS78}. 
For example, when $n=3$ and $h(r)=r^{|\log r|}$ for sufficiently small  $r$, the condition \eqref{eq:74} is satisfied. When $n>3$ and $h(r)=r|\log r|^{-\gamma}$ with $\gamma>1/(n-3)$ for sufficiently small $r$, \eqref{eq:73} is satisfied.  

Assume that $T_h$ satisfies either \eqref{eq:74} or \eqref{eq:73}. Then
\begin{equation}\label{eq:75}
    \lim_{t\downarrow 0}\bP^\mathbf{o}[\sigma_{T_h}\leq t]=\bP^\mathbf{o}[\sigma_{T_h}=0]=0. 
\end{equation}
On the other hand, since $\widetilde d(\mathbf{o},T_h)=0$,  Theorem~\ref{THM61} yields
\[
    \lim_{t\downarrow 0}\left(-t\cdot \log \bP^\mathbf{o}[\sigma_{T_h}\leq t]\right)=0. 
\]
Consequently, the convergence in \eqref{eq:75} is slower than $\ee^{-\gamma/t}$ for every $\gamma>0$ in the sense that $$\lim_{t\downarrow 0}\frac{\ee^{-\gamma/t}}{\bP^\mathbf{o}[\sigma_{T_h}\leq t]}=0.$$

\subsection{Small-time decay rate of hitting time distribution}

In the previous example, we showed that the decay rate of the hitting time distribution can be slower than the classic exponential order $\re^{-\gamma/t}$ for any $\gamma>0$. In this subsection, we further illustrate that the small-time decay rate of \(\bP^x[\sigma_B\le t]\) is characterized by the modified distance $\tilde d(x,B)$. In particular, a larger value of $\tilde d(x,B)$ leads to a faster decay as $t\downarrow 0$.

\begin{prop}\label{PRO73}
Assume that the assumptions of Theorems~\ref{THM61-2} and \ref{THM61} are both satisfied, and that the transition density function of $X$ satisfies the two-sided sub-Gaussian estimate \eqref{eq:71}. Let $x\in E$ and $B\in \sB^n(E)$ satisfy $\tilde{d}(x,B)>0$. Then
 \begin{equation}\label{eq:718}
     \lim_{t\downarrow 0}\frac{\bP^x[\sigma_B\leq t]}{\exp\left\{-\gamma_1 t^{-\frac{1}{\beta -1}}\right\}}=\lim_{t\downarrow 0}\frac{\exp\left\{-\gamma_2 t^{-\frac{1}{\beta -1}}\right\}}{\bP^x[\sigma_B\leq t]}=0
 \end{equation}
 for any constants $\gamma_1,\gamma_2>0$ such that
\[
    \gamma_1<c_2\cdot \widetilde{d}(x,B)^{\beta/(\beta-1)}\leq C_2\cdot \widetilde{d}(x,B)^{\beta/(\beta-1)}<\gamma_2.
\]
In particular, if $x_1,x_2\in E$ and $B_1,B_2\in \sB^n(E)$ satisfy
\begin{equation}\label{eq:713}
    c_2\cdot \widetilde{d}(x_1,B_1)^{\beta/(\beta-1)}>C_2\cdot \widetilde{d}(x_2,B_2)^{\beta/(\beta-1)},
\end{equation}
then
\begin{equation}\label{eq:714}
    \lim_{t\downarrow 0}\frac{\bP^{x_1}[\sigma_{B_1}\leq t]}{\bP^{x_2}[\sigma_{B_2}\leq t]}=0. 
\end{equation}
\end{prop}
\begin{proof}
  Fix $x\in E$ and $B\in \sB^n(E)$ with $\tilde{d}(x,B)>0$.  Choose $\delta>0$ sufficiently small such that 
  \[
      \gamma_\delta:= c_2(1-\delta) \tilde{d}(x,B)^{\beta/(\beta-1)}>\gamma_1. 
  \]
  It follows from \eqref{eq:33} that
  \[
      \limsup_{t\downarrow 0} \frac{\bP^x[\sigma_B\leq t]}{\exp\left\{-\gamma_1 t^{-\frac{1}{\beta -1}}\right\}}\leq C_3\limsup_{t\downarrow 0} \frac{\exp\left\{-\gamma_\delta t^{-\frac{1}{\beta -1}}\right\} }{\exp\left\{-\gamma_1 t^{-\frac{1}{\beta -1}}\right\}}=0.
  \]
  Therefore, 
  \[
      \lim_{t\downarrow 0}\frac{\bP^x[\sigma_B\leq t]}{\exp\left\{-\gamma_1 t^{-\frac{1}{\beta -1}}\right\}}=0.
  \] 
  The second equality follows analogously by applying \eqref{eq:317}. 

  Next, let $x_1,x_2\in E$ and $B_1,B_2\in \sB^n(E)$ satisfy \eqref{eq:713}. Choose a constant $\gamma$ such that
 \[
   C_2\cdot \widetilde{d}(x_2,B_2)<\gamma< c_2\cdot \widetilde{d}(x_1,B_1).
 \]
Then
 \[
     \frac{\bP^{x_1}[\sigma_{B_1}\leq t]}{\bP^{x_2}[\sigma_{B_2}\leq t]}=\frac{\bP^{x_1}[\sigma_{B_1}\leq t]}{\exp\left\{-\gamma t^{-\frac{1}{\beta -1}}\right\}   }\cdot \frac{\exp\left\{-\gamma t^{-\frac{1}{\beta -1}}\right\}  }{\bP^{x_2}[\sigma_{B_2}\leq t]}.
 \]
The desired result follows by applying the first part of the theorem to the two factors on the right-hand side. (Notice that the second equality in \eqref{eq:718} remains valid even when $\tilde{d}(x,B)=0$.) 
\end{proof}

If the exponential constants $c_2$ and $C_2$ in \eqref{eq:71} coincide, namely, 
\[
    p(t,x,y)\asymp \frac{1}{t^{\alpha/\beta}}\exp\left\{-C_2 \left(\frac{d(x,y)}{t^{1/\beta}}\right)^{\beta/(\beta-1)} \right\},
\]
then the second conclusion of Proposition~\ref{PRO73} implies that $\bP^{x_1}[\sigma_{B_1}\leq t]$ decays faster than $\bP^{x_2}(\sigma_{B_2}\leq t)$ as $t\downarrow 0$, whenever $\tilde{d}(x_1,B_1)>\tilde{d}(x_2,B_2)$. 
(Note that if $x_2\in B_2^r$, which indicates that
$\widetilde d(x_2,B_2)=0$, then $\bP^{x_2}[\sigma_{B_2}\leq t]\equiv 1$
does not decay to zero. Nevertheless,
the conclusion \eqref{eq:714} remains valid in this case as well.) In particular, the following example recovers one of the main results in \cite{BBM}.

\begin{exe}
Consider the Brownian motion $B=(B_t)_{t\geq 0}$ on $\bR^n$, with $n\geq 2$, equipped with the Euclidean metric $d(x,y):=|x-y|$.  Let $U,W$ be two (connected) open subsets of $\bR^n$, both containing the origin $\mathbf{o}$. Recall that $\tau_U:=\inf\{t>0:B_t\notin U\}$ denotes the first exit time of $U$, $\partial U:=\overline{U}\setminus U$ is the boundary of $U$ and $(\partial U)^r$ represents the the set of all regular points for $\partial U$. The corresponding notation for $W$ is defined similarly. 

In \cite[Theorem~1.1]{BBM}, it was proved that if $d(\mathbf{o},(\partial U)^r)<d(\mathbf{o},(\partial W)^r)$, then
\begin{equation}\label{eq:716}
    \lim_{t\downarrow 0}\frac{\bP^\mathbf{o}(\tau_U<t)}{\bP^\mathbf{o}(\tau_W<t)}=\infty. 
\end{equation}

We  show that this result is an immediate consequence of Proposition~\ref{PRO73}. 
Indeed, since Brownian motion is conservative, i.e., $\zeta\equiv \infty$, and has continuous sample paths, we have
 \[
     \tau_U=\sigma_{\partial U}=\sigma_{(\partial U)^r},\quad \bP^\bo\text{-a.s.},
 \]
 where the second equality follows from the fact that $\partial U\setminus (\partial U)^r$ is polar. Similarly, 
 \[
     \tau_W=\sigma_{\partial W}=\sigma_{(\partial W)^r},\quad \bP^\bo\text{-a.s.}
 \]
Moreover,  \eqref{eq:714} remains valid if ``$\leq t$" is replaced by ``$<t$", since the estimates \eqref{eq:33} and \eqref{eq:317} also hold with this replacement. Therefore, it suffices to verify that
\begin{equation}\label{eq:717}
    \tilde{d}(\mathbf{o},(\partial U)^r)<\tilde{d}(\mathbf{o},(\partial W)^r),
\end{equation}
and then apply Proposition~\ref{PRO73}. By Proposition~\ref{PRO26}, we have
\[
    \widetilde{d}(\bo, (\partial U)^r)=\widetilde{d}(\bo, \partial U)=d(\bo, (\partial U)^r),
\]
and the same identity holds for  $W$. Combining these identities with the assumption $d(\mathbf{o},(\partial U)^r)<d(\mathbf{o},(\partial W)^r)$ gives \eqref{eq:717}.
\end{exe}

\appendix


 \section{Integral representation of hitting probabilities under duality assumption}

 Consider the Borel right process $X$ on $E$ that satisfies the assumptions (AC), (H2a) and (H2b) in Section \ref{SEC5}.  Without loss of generality, we assume that $X$ is transient. 
 For $f\in \pp\sB(E)$, define $\widehat{U} f(y):=\int_E f(x)u (x,y)\mu(\dd x)$. 
 Under the duality assumption (H2a), the following identity holds: 
 \begin{equation}\label{eq:B1-1}
    \bE^x[u(X_{\sigma_B},y)]=\widehat\bE^y[u(x, \widehat{X}_{\widehat \sigma_B})],\quad \forall x,y\in E,
 \end{equation}
 where $\widehat{\IE}^y$ denotes the expectation with respect to the law of the dual process $\widehat{X}$, and $\widehat{\sigma}_B$ is the first hitting  time of $B\in\sB^n(E)$ with respect to $\widehat{X}$;
 see \cite[Chapter VI, (1.16)]{BG68} and the remark below \cite[Chapter VI, (1.20)]{BG68}. For a positive measure $\nu$ on $E$, define $U\nu(x):=\int_E u(x,y)\nu(\dd y)$ and $\widehat{U}\nu(y):=\int_E u(x,y)\nu(\dd x)$. 

 \begin{lem}\label{LMB1}
 Let $B\in \sB^n(E)$ be such that $\overline{B}$ is compact.
Then there exists a finite measure $\nu_B$ concentrated on $\overline{B}$ such that 
\[
    \bP^x[\sigma_B<\infty]=\int_E u(x,y)\nu_B(\dd y),\quad \forall x\in E. 
\]
 \end{lem}
 \begin{proof}
     Note that $\bP^x[\sigma_B<\infty]=P_B \mathbf{1}_E(x)$. Let $g_n:=n(\mathbf{1}_E-nU_n\mathbf{1}_E)\in \sB(E)$. Then $Ug_n\uparrow 1_E$ pointwise; see, e.g., \cite[Chapter II, (2.6)]{BG68}. Define
     \[
         \nu_n(\cdot):=\int_E g_n(x)\mu(\dd x)\widehat{\bE}^x\left[\widehat{X}_{\widehat\sigma_B}\in \cdot \right],
     \]
     which is a positive measure concentrated on $\overline{B}$, since $\widehat{X}_{\widehat \sigma_B}\in B\cup {}^rB\subset \overline{B}$, $\widehat{\bP}^x$-a.s., where ${}^rB$ denotes the set of all regular points for $B$ with repsect to $\widehat{X}$. It follows from \eqref{eq:B1-1} that 
     \begin{equation}\label{eq:B3-2}
         U\nu_n(x)=P_B(Ug_n)(x) \uparrow P_B1_E(x),\quad n\rightarrow \infty. 
     \end{equation}

     We claim that 
     \begin{equation}\label{eq:B2-2}
       \sup_{n\geq 1}\nu_n(E)=  \sup_{n\geq 1} \nu_n(\overline{B})<\infty. 
     \end{equation}
    Indeed, take an open set $G\supset \overline{B}$ and $f\in \pp C_c(E)$ such that $\overline{G}$ is compact and $f|_{\overline{G}}=1$. (Since $E$ is locally compact, such $G$ and $f$ always exist.) Note that  $\widehat{U}f(y)>0$ for any $y\in G$, due to the right continuity of $\widehat{X}$. By assumption (H2b), $\widehat{U}f$ is continuous. Therefore,
    \[
        \epsilon:=\inf_{y\in \overline{B}}\widehat{U}f(y)>0. 
    \]
    It follows that 
    \[
        \epsilon\cdot  \nu_n(\overline{B})\leq \int_E \widehat{U}f(y)\nu_n(\dd y)=\int_E U\nu_n(x)f(x)\mu(\dd x)\leq \int_E  P_B\mathbf{1}_E(x) f(x)\mu(\dd x)\leq \int_E f(x)\mu(dx)<\infty. 
    \]
    This proves \eqref{eq:B2-2}. 

    Note that $\overline{B}$ equipped with the metric $d$ is a compact metric space, and $\{\nu_n\}$ is a uniformly bounded sequence of  measures on it. Therefore, $\{\nu_n\}$ admits a weakly convergent subsequence, which is still denoted by $\{\nu_n\}$, and whose limit is denoted by $\nu$. Clearly, $\nu(E)=\nu(\overline{B})=\lim_{n\rightarrow \infty}\nu_n(\overline{B})<\infty$. We aim to show that 
    \[
        P_B\mathbf{1}_E(x)=U\nu(x),\quad \forall x\in E,
    \]
    which gives the desired conclusion with $\nu_B=\nu$. To  this end, take an arbitrary $h\in \pp C_c(E)$. We have
    \[
        \int_E U\nu_n(x)h(x)\mu(\dd x)=\int_E \widehat{U}h(y)\nu_n(\dd y). 
    \]
    Since $\widehat{U}h\in C(E)$ by (H2b), the weak convergence of $\nu_n$ implies that 
    \[
        \lim_{n\rightarrow \infty} \int_E U\nu_n(x)h(x)\mu(\dd x)=\int_E \widehat{U}h(y)\nu(\dd y)=\int_E U\nu(x)h(x)\mu(\dd x). 
    \]
    Using \eqref{eq:B3-2} and applying the monotone convergence theorem, we obtain
    \[
        \int_E P_B\mathbf{1}_E(x)h(x)\mu(\dd x)=\int_E U\nu(x)h(x)\mu(\dd x),\quad \forall h\in \pp C_c(E).
    \]
    Note that $P_B1_E, U\nu\in \sB^*(E)$ admit Borel measurable  $\mu$-a.e. versions. Hence, $P_B1_E=U\nu$, $\mu$-a.e. Since both $P_B1_E$ and $U\nu$ are excessive function, it follows from \cite[Chapter VI, (1.3)]{BG68} that $P_B1_E=U\nu$ pointwise. This completes the proof. 
 \end{proof}

\section*{Acknowledgment}
The authors would like to express their sincere gratitude to Professor Zhen-Qing Chen for his valuable comments and constructive suggestions on this work.

\bibliographystyle{amsalpha}

\bibliography{HitProb}

\vskip 0.3truein

\noindent {\bf Liping Li}

\smallskip \noindent
Fudan University, Shanghai, China

\noindent
E-mail:  \texttt{liliping@fudan.edu.cn}

\vskip 0.3truein

\noindent {\bf Shuwen Lou}

\smallskip \noindent
Department of Mathematics and Statistics,

\noindent
Loyola University Chicago,
Chicago, IL 60660, USA

\noindent
E-mail:  \texttt{slou1@luc.edu}

\end{document}